\documentclass[12pt]{amsart}

\usepackage{amsmath}
\usepackage{amsthm}
\usepackage{xfrac}
\usepackage{txfonts}
\usepackage{indentfirst}
\usepackage{mathrsfs}
\usepackage{mathtools}
\usepackage{amssymb}
\usepackage{dutchcal}

\newtheorem{theorem}{Theorem}[section]
\newtheorem{lemma}{Lemma}[section]

\newtheorem{definition}{Definition}[section]

\newtheorem{proposition}{Proposition}[section]
\newtheorem{corollary}{Corollary}[section]

\newtheorem{remark}{Remark}[section]
\newtheorem{claim}{Claim}
\newtheorem*{mainthm}{Main Theorem}
\newtheorem*{thmA}{Theorem A}

\newtheorem*{thmB'}{Theorem B'}
\newtheorem*{thmB''}{Theorem B''}

\numberwithin{equation}{section}

\newcommand{\R}{\mathbb{R}}
\newcommand{\Z}{\mathbb{Z}}
\newcommand{\eps}{\varepsilon}

\title{A Dichotomy for the dimension of solenoidal attractors on high dimensional space}
\author{Haojie Ren}

\date{\today}
\address{School of Mathematical Sciences, Fudan University, No 220 Handan Road, Shanghai, China 200433}
\email{20110180012@fudan.edu.cn}
\begin{document}
\setlength{\parindent}{1em}

\maketitle

\begin{abstract}
We study dynamical systems generated by skew products:
$$T: [0,1)\times\mathbb{C}\to [0,1)\times\mathbb{C}
\quad\quad T(x,y)=(bx\mod1,\gamma y+\phi(x))$$
where  integer $b\ge2$, $\gamma\in\mathbb{C}$ such that $0<|\gamma|<1$, and $\phi$ is a real analytic $\mathbb{Z}$-periodic function. 
Let $\Delta\in[0,1) $ such that $\gamma=|\gamma|e^{2\pi i\Delta}$.
For the case $\Delta\notin\mathbb{Q}$
we prove the following  dichotomy for the solenoidal attractor $K^{\phi}_{b,\,\gamma}$ for $T$: Either $K^{\phi}_{b,\,\gamma}$ is a graph of  real analytic function, or the Hausdorff dimension of $K^{\phi}_{b,\,\gamma}$ is  equal to $\min\{3,1+\frac{\log b}{\log1/|\gamma|}\}$. Furthermore, given $b$ and $\phi$, the former alternative only  happens for countable many $\gamma$ unless $\phi$ is constant.
\end{abstract}

\section{introduction}
One of main purposes of dynamical system is to study the strange attractors. In this paper, we concern the Hausdorff dimension of solenoidal attractors on skew products.

For any positive integer $d$, let $\mathbf{B}_d:=\{\,z\in\mathbb{R}^d:\,|z|\le 1\}.$
For any functions $f:[0,1)\to[0,1)$ and $g:[0,1)\times\mathbf{B}_d\to\mathbf{B}_d$, define the map 
$$F:\,[0,1)\times\mathbf{B}_d\to[0,1)\times\mathbf{B}_d\quad\quad F(x,\,y)=\big(f(x),\,g(x,y)\big).$$
Thus 
\begin{equation}\label{def:solenoid}
K_{F}:=\bigcap_{n=0}^{\infty} F^n\big(\,[0,1) \times\mathbf{B}_d\,\big)
\end{equation}
is an invariant set and called {\em solenoidal attractor} or  {\em solenoid} for $F$. Denote $\text{dim}_H(K_F)$ as the Hausdorff dimension of $K_F$.

{\bf Historical remarks.}
The study of solenoidal attractors on skew products has long history. For the case $d=1$,
Alexander and Yorke  \cite{Alexander1984fat} introduced a class maps called generalized baker's transformation:
\begin{equation}\nonumber
B: [-1,1]\times[-1,1]\circlearrowleft
\quad\quad B(x,y)=\begin{cases}
(2x-1,\gamma y+(1-\gamma))&\, x\ge0\\
(2x+1,\gamma y-(1-\gamma))&\, x<0.
\end{cases}
\end{equation}
Solomyak \cite{Solo, PS} proved $\text{dim}_H(K_B)=2$ for Lebesgue-a.e. $\gamma\in(1/2,1] $. Later Hochman \cite{hochman2014self} showed the packing dimension of the exception set is zero and for every $\gamma$ transcendental in Varj\'u \cite{Varju1}. See \cite{Varju2} for some $\gamma$ algebraic.

  To generalize $B$ to the nonlinear case, Tsujii \cite{tsujii2001fat} introduced dynamical systems generated by maps:
\begin{equation}\label{Map}
\tilde{T}: [0,1)\times\mathbb{R}\to [0,1)\times\mathbb{R}
\quad\quad \tilde{T}(x,y)=(bx\mod1,\gamma y+\phi(x))
\end{equation}
where  integer $b\ge2$, $0<\gamma<1$ and $\phi$ is a real $\mathbb{Z}$-periodic Lipschitz function. ( Note that  $\tilde{T}\big(\,[0,1)\times\mathbf{B}(0,\,\frac{\parallel\phi\parallel_{\infty}}{1-\gamma}\,)\,\big)\subset [0,1)\times\mathbf{B}(0,\,\frac{\parallel\phi\parallel_{\infty}}{1-\gamma}\,)$, thus we can define the solenoids $K_{\tilde{T}}$ by (\ref{def:solenoid})).
For $\gamma\in(1/b,1)$,
Tsujii  \cite{tsujii2001fat} proved that $\text{dim}(K_{\hat{T}})=2$ for $C^2$ generic $\phi$.
Later the author \cite{Ren2022} gave a complete answer for all $\gamma\in(0,1)$ when $\phi$ is a real analytic $\mathbb{Z}$-periodic function. See e.g. \cite{barany2019hausdorff, shen2018hausdorff, Zhang2018, Rams} for results about the dimension of solenoids when $d=1$.

For the case $d=2$, studying the dimension of solenoids got harder.
A well-known example is Smale-Williams-type solenoids generated by the map
\begin{equation}\nonumber
Q: [0,1)\times\mathbb{R}^2\circlearrowleft
\quad\quad Q(x,y)=(bx\mod1,\gamma y+\phi(x), \lambda y+\psi(x))
\end{equation}
where  integer $b\ge2$, $0<\gamma,\,\lambda<1$ and $\phi,\,\psi$ are real $\mathbb{Z}$-periodic Lipschitz functions. In particular $K_Q$ is called {\em Smale-Williams solenoid } if $\phi(x)=\cos(2\pi x)$, $\psi(x)=\sin(2\pi x)$ and $b=2$.
When  $\gamma,\,\lambda<1/b,$  $K_Q$ is called {\em thin solenoid}. \cite{Schmeling, Simon} gave the positive answers for the dimension of thin Smale-Williams solenoids. Later Simon \cite{Simon2} calculated the the Hausdorff dimension of thin Smale-Williams solenoids when $Q$ is one-to-one map. Recently  Mohammadpour et al.\cite{Przytycki} consider a  general class of  functions $Q'$. Under the transversality assumption, \cite{Przytycki} calculated the the Hausdorff dimension of thin solenoids for $Q'$ when $Q'$ is injective on $K_{Q'}$. More results about the Hausdorff dimension of solenoids can be found in \cite{Silva,Simon2,Mihailescu, Rams2,Rams3}.

{\bf Main findings.} It is well-known that transversality is vital to study the dimension of solenoids. One of our goals in this paper is to find  ways to  remove the transversality assumption
 for the case $d\ge2$. Another goal is to explore ways to use the methods of additive combination\cite{hochman2014self,hochman2021,barany2019hausdorff,Hochman2022} to study the dimension of strange attractors for nonlinear maps as \cite{ren2021dichotomy,Ren2022} did.
  Thus  we 
  introduce the following  nonlinear maps 
 \begin{equation}\label{MapT}
 T: [0,1)\times\mathbb{C}\to [0,1)\times\mathbb{C}
 \quad\quad T(x,y)=(bx\mod1,\gamma y+\phi(x))
 \end{equation}
 where  integer $b\ge2$, $\gamma\in\mathbb{C}$ such that $0<|\gamma|<1$, and $\phi$ is a real $\mathbb{Z}$-periodic Lipschitz function.
  Note that map $T$ is a generalization of $\tilde{T}$, see (\ref{Map}).  Let $K^{\phi}_{b,\,\gamma} $ be the solenoid for $T$.
  For any $\gamma\in\mathbb{C}\setminus\{0\}$, let $\Delta=\Delta(\gamma)\in[0,1)$ be the unique number such that
 \begin{equation}\label{def:Delta}
 \gamma=|\gamma|e^{2\pi i\Delta(\gamma)}.
 \end{equation}
 
 The main result of this paper is following.
 \begin{mainthm} Let $b\ge 2$ be an integer, $\gamma\in\mathbb{C}$ such that $0<|\gamma|<1$. Let $\phi$ be a $\Z$-periodic real analytic function. Then exactly
 	one of the following holds:
 	\begin{enumerate}
 		\item [(i)] $K^{\phi}_{b,\gamma}$ is a graph of a real analytic function;
 		\item [(ii)] $dim_H(K^{\phi}_{b,\gamma})=\min\{3,1+\frac{\log b}{\log1/|\gamma|}\}$
 		when $\Delta\notin\mathbb{Q}$;
 		\item [(iii)] $dim_H(K^{\phi}_{b,\gamma})\ge\min\{2,1+\frac{\log b}{\log1/|\gamma|}\}$
 		when $\Delta\in\mathbb{Q}$.
 	\end{enumerate}
 	Moreover given $b$ and  non-constant $\phi$, the first alternative only holds for countable many $0<|\gamma|<1$.
 \end{mainthm}
\medskip{\bf Organization.} 
In Sect. \ref{ProveM}, we will introduce Theorem A  which is used for proving Main Theorem. The rest parts of this paper is devoted to prove Theorem A.
 In Sect. \ref{pre},
 we will introduce the classical Ledrappier-Young theory for the map $T$, and recall  some basic properties of entropy that serve as the basic tools in this paper.
The main difficulty in proving Theorem A is proving Theorem \ref{thm:entropygrowth}, since the rest parts can be proved by the similar method of \cite{Ren2022} (See Sect.\ref{Theorem A} for explanations). Thus we shall show the dimension conservation property of $m_x$
in Sect. \ref{DC}, which is used to prove Theorem \ref{thm:entropygrowth} in Sect. \ref{sec:inverse}. Finally we will use Theorem \ref{thm:entropygrowth} to give a  sketchy proof for Theorem A in Sect. \ref{sec:proveA}.

\medskip
{\bf Acknowledgment.} We would like to thank  Feliks Przytycki for useful communication.

\section{Theorem A and proof of Main Theorem }\label{ProveM}
In this section, we shall first give the expression of solenoidal attractors $K^{\phi}_{b,\,\gamma}$. Later we will introduce Theorem A and   give some explanations about the ideas of  proof,  which is our main finding in this paper. 
Finally we will use Theorem A to prove the Main Theorem as we did in \cite{Ren2022}.
\subsection{Expression of $K^{\phi}_{b,\,\gamma}$}
Let $\mathbb{Z}_+$ denote the set of positive integers and $\mathbb{R}_+$ denote the set of nonnegative real numbers.
Let $\mathbb{N}$ be the set of nonnegative integers.
Let $\varLambda=\{ 0,1,...,b-1 \}$,
$\varLambda^{\#}=\bigcup_{n=1}^{\infty} \varLambda ^n$,
$\Sigma=\varLambda^{\mathbb{Z}_+}$. For any word $\textbf{j}\in \varLambda^{\#}\cup\Sigma $ let $|\textbf{j}|=t$ if $\textbf{j}=j_1j_2\ldots j_t\in \varLambda^{\#}$, and $|\textbf{j}|=\infty$ if $\textbf{j}\in\Sigma$. 
 For any 
 $x\in [0,1]$ define function
\begin{equation}\label{S}
S(x,\textbf{j})=S_{b,\,\gamma}^\phi(x,\textbf{j})=\sum\limits_{n=1}^{|\textbf{j}|}{\gamma^{n-1}\phi\left(\frac x{b^n}+\frac{j_1}{b^n}+\frac{j_2}{b^{n-1}} + \cdot \cdot \cdot + \frac{j_n}b\right)},
\end{equation}
thus we have
 $$K^{\phi}_{b,\,\gamma}=\bigg\{(x, S(x,\textbf{j})):x\in [0,1),\,\textbf{j}\in\Sigma\bigg\}.$$
 See \cite{tsujii2001fat} for details, the proof in \cite[ Sect. 2]{tsujii2001fat} works for all $\gamma\in\mathbb{C}$ such that $|\gamma|\in(0,1)$.
 
 It is necessary to recall the following formulas from \cite{tsujii2001fat}, which serve as the  basic tools in our paper.
  For any $x\in [0,1]$, $\textbf{w}=w_1w_2\ldots w_m\in\varLambda^{\#}$ and each $\textbf{i},\,\textbf{j}\in\varLambda^{\#}\cup\Sigma,$ 
  let 
  \begin{equation}\label{eq:symbolfunction}
  \textbf{w}(x):=\frac{x+w_1+\ldots+w_{m}b^{m-1}}{b^m}
  \end{equation}
  and 
  \begin{equation}\label{eq:symbolsum}
  \textbf{w}\textbf{i}=w_1\ldots w_m i_1i_2\ldots
  \end{equation}
   as usual, so we have
 \begin{equation}\label{F}
 S(x,\textbf{w}\textbf{i})=S(x,\textbf{w})+\gamma^{m}S(\textbf{w}(x),\textbf{i})
 \end{equation}
 by the definition of function $S(\cdot,\cdot)$.
 Therefore we have
 \begin{equation}\label{FM}
 S(x,\,\textbf{w}\textbf{i})-S(x,\,\textbf{w}\textbf{j})=\gamma^m\bigg(S(\textbf{w}(x),\,\textbf{i})-S(\textbf{w}(x),\,\textbf{j})\bigg).
 \end{equation}
\subsection{Theorem A}\label{Theorem A}
Let $\nu$ denote even distributed probability  measure on $\varLambda$ and let $\nu^{\mathbb{Z}_+}$ be product measure on $\Sigma.$  Denote the Lebesgue measure on $[0,1)$ as $\mathcal{m}$. Define the probability measure on $[0,1)\times\mathbb{C}$
$$\omega:=G( \mathcal{m} \times\nu^{\mathbb{Z}_+})$$
where the map
\begin{equation}\nonumber
G : [0,1)\times\Sigma\to [0,1)\times\mathbb{C}\quad\quad G(x,\textbf{j})=(x, S(x,\textbf{j})).
\end{equation}
In fact the measure $\omega$ is the SRB measure for the map $T$.
For every  $x\in [0,1]$ define the map
\begin{equation}\label{S_x}
S_x: \Sigma\to\mathbb{C}\quad\quad S_x(\,\textbf{j}\,)=S(x,\textbf{j})
\end{equation}
and  measure
\begin{equation}\label{def:m_x}
m_x:=S_x(\nu^{\mathbb{Z}_+})
\end{equation}
 as in \cite{tsujii2001fat, Ren2022}.
Note that $\omega=\int_{[0,1) }\big(\delta_x\times m_x\big)\,dx.$
 For any $n\ge1$, we have
 \begin{equation}\label{FundementalFormular}
m_x=\frac1{b^n}\sum_{\textbf{j}\in\varLambda^n} T^n(m_{\textbf{j}(x)})
\end{equation}
where probability measures
\begin{equation}\label{T^nm_x}
T^n(m_{\textbf{j}\,(x)})=f_{x,\,\textbf{j}}(m_{\textbf{j}(x)})
\end{equation}
and $f_{x,\,\textbf{j}}(y)=\gamma^ny+S(x,\,\textbf{j}),\,\forall\, y\in\mathbb{C}$. See  \cite[Section 2]{tsujii2001fat} for details.

A probability measure  $\mu$ in a metric space $X$ is called {\em exact-dimensional} if there exists a constant $\kappa\ge 0$ such that for $\mu-\text{a.e.}$ $x$,
\begin{equation}
\lim_{r\to0}\frac{\log\mu\big(\mathbf{B}(x,r)\big) } {\log r}=\kappa.
\end{equation} 
In this situation, we write $\text{dim}(\mu)=\kappa$ and call it {\em the dimension of $\mu$}.

To study the dimension of the SRB measure $\omega$ as \cite{Ren2022} did, we shall recall the following gentle  transversality condition (H).
  Shen and the author \cite{ren2021dichotomy}  put it and studied the case $\gamma\in(1/b,1)$.  Recently Gao and Shen \cite{gao2022} studied the case $\gamma\in\mathbb{C}$ such that $0<|\gamma|<1$.
\begin{definition}
	Given an integer $b\ge 2$ and $\gamma\in\mathbb{C}$ such that $0<|\gamma|<1$,	we say that a real $\mathbb{Z}$-periodic $C^1$ function $\phi(x)$ satisfies 
	the condition (H) if 
	$$S(x,\textbf{j})-S(x,\textbf{i}) \not\equiv 0, \quad \forall \, \textbf{j} \neq \textbf{i} \in \Sigma.$$
\end{definition}
The following theorem is our main findings in this paper.
\begin{thmA}
	Let $\phi(x)$ be  a real analytic $\mathbb{Z}$-periodic function  satisfies the condition (H) for an integer $b\ge 2$,  $\gamma\in\mathbb{C}$ such that $0<|\gamma|<1$.
	If $\Delta\in\mathbb{R}\setminus\mathbb{Q}$,
	 then
	$$dim (\omega)=\min\{1+\frac{\log b}{\log1/|\gamma|},3\}.$$
\end{thmA}
Recall that $\Delta$ is defined by (\ref{def:Delta}).

{\bf Explanations of Theorem A.}
Inspired by the  recent works \cite{hochman2014self,barany2019hausdorff,ren2021dichotomy}, the author \cite{Ren2022} has studied Theorem A for the case $\gamma\in(0,1)$, thus it is nature to consider whether these methods can be used to study the general case $\gamma\in\mathbb{C}$. In fact,  most parts of proof in  \cite{Ren2022} also work for  the case $\gamma\in\mathbb{C}$ such that $|\gamma|\in(0,1)$. Therefore, if the inverse theorem for convolution on $\mathbb{R}$ of Hochman \cite[Theorem 2.7]{hochman2014self} also holds for all  measures in $\mathscr{P}(\mathbb{C})$, there will be little difficulty to prove Theorem A. Unfortunately that is not true, studying the entropy of convolutions on $\mathbb{R}^d$ gets more tough.

To study the selfsimilar measures on $\mathbb{R}^d$, Hochman \cite{hochman2021} generalized  the inverse theorem on $\mathbb{R}$ \cite[Theorem 2.7]{hochman2014self} to high-dimension spaces \cite[Theorem 2.8]{hochman2021}. However the latter is not easy to use directly. Thus the main difficulty in our paper is to give a  corollary of \cite[Theorem 2.8]{hochman2021}, whose condition is easier to verify in our system (See Theorem \ref{thm:entropygrowth}). 

Indeed,
The  classical Ledrappier-Young theory implies that:
 for Lebesgue-a.e. $x\in[0,1]$
$$\text{dim}(m_x)=\alpha.$$
Thus it suffice to prove $\alpha=\min\{2,\frac{\log b}{\log1/|\gamma|}\}$ (See Theorem \ref{TheoremB} for details).

As mentioned before, studying the  dimension of measures $m_x$ directly is  hard, thus we hope to understand the properties of   projection measures $\pi_{\theta}m_x$ and the conditional measures $\{\,(m_x)_{\pi_{\theta}^{-1}(z)}\}_{z\in\mathbb{R}}$, see Definition \ref{def:m}. Thanks to the methods in \cite{Feng, Falconer}, we  build a new Ledrappier-Young theory for the  nonlinear system generated by the map $T$, which establishes the connections between $\text{dim}(m_x)$ and $\text{dim}(\pi_{\theta}m_x)$,  $\text{dim}\big((m_x)_{\pi_{\theta}^{-1}(z)}\big)$ for the case $\Delta\notin\mathbb{Q}$,
see Theorem \ref{thm:DC}. Therefore inspired by the work of Hochman \cite{hochman2021},
  we can  use the skills in \cite{hochman2014self, hochman2021,barany2019hausdorff,Ren2022} to  build Theorem \ref{thm:entropygrowth} which  is a  corollary of \cite[Theorem 2.8]{hochman2021}. The rest of proofs are similar to \cite{Ren2022} and we will omit some details for convenience.
\subsection{Proof of Main Theorem}
We shall first introduce the following Theorem \ref{thm:dichotomy}, which is an immediate consequence of \cite[Theorem 2.1]{gao2022}. Note that the proof of \cite[Theorem 2.1]{Ren2022} also works for  $\gamma\in\mathbb{C}$ such that $0<|\gamma|<1$, thus we omit the proof of Theorem \ref{thm:dichotomy}.
\begin{theorem}\label{thm:dichotomy}
	Fix $b\ge 2$ integer and $\gamma\in\mathbb{C}$ such that $0<|\gamma|<1$. 
	 Assume that $\phi$ is analytic $\Z$-periodic function. Then exactly one of the following holds:
	\begin{enumerate}
		\item [(H.1)]  $K^{\phi}_{b,\,\gamma}$ is a graph of  real analytic function;
		\item [(H.2)]  $\phi$ satisfies the condition (H).
	\end{enumerate}
\end{theorem}
For any $\theta\in\mathbb{R}$, let $\pi_{\theta}:\mathbb{C}\to\mathbb{R}$ be the projection function such that
\begin{equation}\label{def:pi}
z\mapsto \text{Re}(ze^{-2\pi i\theta})
\end{equation}
where $\text{Re}(ze^{-2\pi i\theta})$ is the real part of $ze^{-2\pi i\theta}$. 
\begin{proof}[Proof of the Main Theorem] 
If $\phi$ is a non-constant function, then there exists $m\in\mathbb{Z}_+$ and $x'\in[0,1) $ such that $\phi(\frac{x'}{b^m})\neq\phi(\frac{x'+1}{b^m})$.
 Let $10_{\infty}=100\ldots,\,0_{\infty}=00\ldots\in\Sigma$. Thus the function $ g(\gamma):=S^{\phi}_{b,\,\gamma}(x',10_{\infty})-S^{\phi}_{b,\,\gamma}(x',0_{\infty}) $
 has at most countable zero points in $\mathbf{B}(0,1)$, which implies: (i) holds for at most countable  many $\gamma\in\mathbb{C}$ such that $0<|\gamma|<1$ and  we only need to consider the case that  condition (H) holds by Theorem \ref{thm:dichotomy}.
 
 For the case  $\Delta\in\mathbb{R}\setminus\mathbb{Q}$, it suffice to show $dim (\omega)=\min\{1+\frac{\log b}{\log1/|\gamma|},3\}$ since $\text{supp}(\omega)=K^{\phi}_{b,\,\gamma}$ and $\text{dim}_B(K^{\phi}_{b,\,\gamma})\le\min\{1+\frac{\log b}{\log1/|\gamma|},3\}.$
 Thus we have (ii) holds by Theorem A.
 
 For the case $\Delta\in\mathbb{Q}$, we may assume that $\Delta\neq0$ by \cite[Main Theorem]{Ren2022}. Let $\Delta=k/m\in(0,1)$ for some $k,\,m\in\mathbb{Z}_+.$ Since 
  the condition (H) holds, there exist $\theta\in[0,1) $ such that 
  \begin{equation}\label{theta}
  \pi_{\theta}\circ S^{\phi}_{b,\,\gamma}(x,1_{\infty})-\pi_{\theta}\circ S^{\phi}_{b,\,\gamma}(x,0_{\infty})\not\equiv0.
  \end{equation}
 define the  function $\overline{\phi}:\mathbb{R}\to\mathbb{R} $ be such that
 $$\overline{\phi}(x)=\pi_{\theta}\bigg(\,\sum_{t=0}^{m-1}\gamma^{m-1-t}\phi(b^tx)\,\bigg).$$
 Thus (\ref{theta}) implies $\overline{\phi}$ is a real analytic $\Z$-periodic function such that the condition (H) holds for  $b^m$ and $|\gamma|^m.$ Therefore we have
 \begin{equation}\label{eq:low0}
 \text{dim}\big(\,\omega^{\overline{\phi}}_{b^m,\,|\gamma|^m}\,\big)=\min\{2,1+\frac{\log b}{\log1/|\gamma|}\}
 \end{equation}
 by \cite[Theorem A]{Ren2022}. Define the function $\hat{\pi}_{\theta}:[0,1)\times\mathbb{C}\to[0,1)\times\mathbb{C}$ such that
 $\hat{\pi}_{\theta}(x,z)=(x,\,\pi_{\theta}(z))$. Therefore we have
 $$\hat{\pi}_{\theta}(\,\omega^{\phi}_{b,\,\gamma}\,)=\omega^{\overline{\phi}}_{b^m,\,|\gamma|^m},$$
 which implies $\text{dim}\big(\,\omega^{\phi}_{b,\,\gamma}\,\big)\ge\min\{2,1+\frac{\log b}{\log1/|\gamma|}\}$ by $\text{dim}\big(\,\omega^{\phi}_{b,\,\gamma}\,\big)\ge\text{dim}\big(\,\hat{\pi}_{\theta}(\omega^{\phi}_{b,\,\gamma})\,\big)$ and (\ref{eq:low0}). Therefore (iii) holds.
\end{proof}

\section{Ledrappier-Young theory and entropy}\label{pre}
In this section we shall first recall the classical Ledrappier-Young formula for $T$.
Then we will give a brief introduce the conditional measure theorem of Rohlin \cite{rokhlin1949fundamental} and recall the definition of entropy. Finally we will introduce some basic properties of the entropy of measures, which serve as basic tools in our paper.
\subsection{Ledrappier-Young formula}\label{ledrappier-young}
The following Theorem \ref{TheoremB} is the classical Ledrappier-Young formula for SRB measure $\omega$, which can be proved by exactly the same methods in \cite[Theorem 3.1]{Ren2022}. Thus we  omit the proof, see also \cite{ledrappier1992dimension, ledrappier1985metric, Shu2010}.
\begin{theorem}\label{TheoremB}
	 Let $b\ge 2$ be an integer and $\gamma\in\mathbb{C}$ such that $0<|\gamma|<1$.
	If $\phi:\mathbb{R}\to\mathbb{R}$ is a $\mathbb{Z}$-periodic Lipschitz function, then
	\begin{enumerate}
		\item[(1)] $\omega$ is exact dimensional;
		\item[(2)] there is a constant $\alpha\in [0,2]$ such that for Lebesgue-a.e. $x\in [0,1) $, $m_x$ is exact dimensional and $\dim(m_x)=\alpha$ and
		\begin{equation}\label{LedrapperYoungF}
		\dim(\omega)=1+\alpha.
		\end{equation}
	\end{enumerate}
\end{theorem}
Recall that the definition of $m_x$ is from (\ref{def:m_x}).

\subsection{Conditional measure and entropy}\label{subsection2}
Let $(\Omega,\mathscr{B},\mu)$ be a Lebesgue space (see \cite[Definition 2.3]{GTM79}). 
A partition $\eta$ of $\Omega$ is called {\em measurable partition} if, up to a  set of measure zero, the quotient space is separated by a countable number of measurable sets $\{	B_i\}$. Let $\hat{\eta}$ be the sub-$\sigma$-algebra of $\mathscr{B}$
whose elements are unions of elements of $\eta$.
\begin{theorem}[Rohlin \cite{rokhlin1949fundamental}]\label{Rohlin}
	Let $\eta$ be a measurable partition of a Lebesgue space $(\Omega,\mathscr{B},\mu)$. Then for every $x$ in the set of full $\mu$-measure, there is a probability measure $\mu^{\eta}_x$ defined on $\eta(x)$, the element of $\eta$ containing $x$. These measures are uniquely characterized (up to set of $\mu$-measure 0) by the following properties: if $A\subset \Omega$ is a measurable set, then $x\mapsto \mu^{\eta}_x(A)$ is $\hat{\eta}$-measurable and $$\mu(A)=\int \mu^{\eta}_x(A) d\mu(x).$$
	These  properties imply that for any $g\in L^1(\Omega,\mathscr{B},\mu)$, $\mu^{\eta}_x(g)=\mathbf{E}_{\mu}(g|\,\hat{\eta})(x)$ for $\mu-\text{a.e.\,}x$, and
	$\mu(g)=\int \mathbf{E}_{\mu}(g|\,\hat{\eta}) d\mu.$
\end{theorem}
In  Theorem \ref{Rohlin}, we call $\{\,\mu^{\eta}_x\,\}_{x\in\Omega}$ {\em the canonical system of conditional measure with $\eta$}. 
A {\em countable partition} $\mathcal{P}$ is a countable collection of pairwise disjoint measurable subsets of $\Omega$ whose union is equal to $\Omega$.
For any sub-$\sigma$-algebra $\mathcal{M}$ of $\mathscr{B}$, any countable $\mathscr{B}$-measurable partition $\mathcal{P}$ of $\Omega$ we define the conditional information
$$\mathbf{I}_{\mu}(\,\mathcal{P}\,|\,\mathcal{M})=-\sum_{I\in\mathcal{P}}\chi_{I}\text{log}_b\mathbf{E}_{\mu}(\chi_{I}|\mathcal{M}),$$
and the conditional entropy
$$H_{\mu}(\,\mathcal{P}\,|\,\mathcal{M})=\int_{\Omega}\mathbf{I}_{\mu}(\,\mathcal{P}\,|\,\mathcal{M})d\mu.$$
If $\mathcal{M}=\{\emptyset,\,\Omega\}$, let $\mathbf{I}_{\mu}(\mathcal{P})=\mathbf{I}_{\mu}(\,\mathcal{P}\,|\,\mathcal{M})$ and
  $H(\mu,\mathcal{P})=H_{\mu}(\,\mathcal{P}\,|\,\mathcal{M})$  for convenience.
 
 For any countable $\mathscr{B}$-measurable partition $\mathcal{Q}$ of $\Omega$.
   Let $\mathcal{Q}(x)$  be the member of $\mathcal{Q}$ that contains $x$. If $\mu(\mathcal{Q}(x))>0$, we call the conditional measure
  $$\mu_{\mathcal{Q}(x)}(A)=\frac{\mu(A\cap \mathcal{Q}(x))}{\mu(\mathcal{Q}(x))}$$
  a {\em $\mathcal{Q}$-component} of $\mu$. 
  Let function $f_b:\mathbb{R}_+\to\mathbb{R}_+$ be a function such that
  $$f_b(p)= -p\log_b p$$
   where  the common convention $0\log 0=0$ is adopted.
  In particular we have
  $$H(\mu, \mathcal{Q})=\sum_{Q\in\mathcal{Q}}f_b\big(\mu(Q)\big).$$
  For another countable partition $\mathcal{P}$, we also have
  $$H(\mu, \mathcal{Q}|\mathcal{P})=\sum_{P\in \mathcal{P},\,\mu(P)>0} \mu(P) H(\mu_P, \mathcal{Q}).$$
  When $\mathcal{Q}$ is a {\em refinement} of $\mathcal{P}$, i.e.,
  $\mathcal{Q}(x)\subset\mathcal{P}(x)$ for every $x\in \Omega$, we have
  $$H(\mu, \mathcal{Q}|\mathcal{P})=H(\mu, \mathcal{Q})-H(\mu, \mathcal{P}).$$

\subsection{Entropy of measures}\label{sub:entropy}
 In this subsection we shall introduce some notations whose were  used extensively in ~\cite{hochman2014self,barany2019hausdorff,Hochman2022,hochman2021,Ren2022}, which also serve as the basic language in this paper. Latter we will recall some basic facts about the Entropy of measures.
 
Let  $(\Omega,\mathscr{B},\mu)$ be a Lebesgue space.
If there exists a sequence of partitions $\mathcal{Q}_i$, $i=0,1,2,\cdots$, such that $\mathcal{Q}_{i+1}$ is a refinement of $\mathcal{Q}_i$, we shall denote $\mu_{x,\,i}=\mu_{\mathcal{Q}_i(x)}$, and call it a {\em $i$-th component measure of  $\mu$}. For a finite set $I$ of integers, if for every $i\in I$, there is a random variable $Y_i$ defined over $(\Omega, \hat{\mathcal{Q}_i}, \mu)$, see Section \ref{subsection2} for the definition of $\hat{\mathcal{Q}_i}$. Then we shall use the following notation
$$\mathbb{P}_{i\in I} (K_i)=\mathbb{P}_{i\in I}^{\,\mu}(K_i):=\frac{1}{\# I} \sum_{i\in I} \mu(K_i),$$
where $K_i$ is an event for $Y_i$. If $Y_i$'s are $\mathbb{R}$-valued random variable, we shall also use the notation
$$\mathbb{E}_{i\in I} (Y_i)=\mathbb{E}^{\,\mu}_{i\in I}(Y_i):=\frac{1}{\# I} \sum_{i\in I} \mathbb{E}(Y_i).$$
Therefore the following holds
$$H(\mu, \mathcal{Q}_{m+n}|\mathcal{Q}_n)=\mathbb{E}(H(\mu_{x,\, n}, \mathcal{Q}_{m+n}))=\mathbb{E}^{\,\mu}_{i=n} (H(\mu_{x,\,i},\mathcal{Q}_{i+m})).$$

In most  cases of this paper, we shall consider the situation  that $\Omega=\mathbb{R},\,\mathbb{C}$. 
Write $\mathbb{R}^d$ for $\mathbb{R}$ or $\mathbb{C}$. 
Let $\mathcal{L}^{\mathbb{R}}_n$ be the partition of $\mathbb{R}$ into $b$-adic intervals of level $n$, i.e., the intervals $[j/b^n, (j+1)/b^n)$, $j\in \mathbb{Z}$.
In this paper we also regard $\mathbb{C}$ as $\mathbb{R}^2$. 
Similar let $$\mathcal{L}^{\mathbb{C}}_n:=\bigg\{\,[j_1/b^n, (j_1+1)/b^n)\times[j_2/b^n, (j_2+1)/b^n)\,\bigg\}_{j_1,\,j_2\in\mathbb{Z}}.$$

In the rest of paper, we may  write $\mathcal{L}^{\mathbb{R}}_n,\,\mathcal{L}^{\mathbb{C}}_n$ as $\mathcal{L}_n$ for convenience. 
For any $I\in\mathcal{L}_n$ and $\mu(I)>0$ let
$$\mu^I:=g_I (\mu_I)$$
where map $g_I(z):=b^n\cdot z+c_I,\,\forall \,z\in\mathbb{R}^d$ and $g_I(I)=[0,1)^d$ holds
for some $c_I\in\mathbb{R}^d$.
As \cite{hochman2014self} did, denote
\begin{equation}\label{def:measureupper}
\mu^{x,\,i}=\mu^{\mathcal{L}_i(x)}.
\end{equation}

 Let $\mathscr{P}(\R^d)$ denote the collection of all Borel probability measures in $\R^d$. 
If a probability measure $\mu\in \mathscr{P}(\mathbb{R}^d)$ is exact dimensional, its dimension is closely related to the entropy. In fact we have the
following result
\cite[Theorem 4.4]{Young}. See also~\cite[Theorem 1.3]{Fan2002}.
\begin{proposition}\label{prop:Young}
	If $\mu \in \mathscr{P}(\mathbb{R}^d)$ is exact dimensional, then
	$$\dim(\mu)= \lim\limits_{n\to\infty}\frac{1}{n} H(\mu,\mathcal{L}_n).$$
\end{proposition}

The following are some well-known facts about entropy and conditional entropy, which will  be used a lot in our work.
See \cite[Section 3.1]{hochman2014self} for details.
Define function
\begin{equation}\label{def:H}
H:[0,1]\to\mathbb{R}_+\quad\quad H(p)=f_b(t)+f_b(1-t).
\end{equation}
\begin{lemma}[Concavity and convexity]\label{lem:concave}
	Consider a measurable space $(\Omega, \mathscr{B})$ which is endowed with partitions $\mathcal{Q}$ and $\mathcal{P}$ such that $\mathcal{P}$ is a refinement of $\mathcal{Q}$. Let $\mu, \mu'$ be probability measures in $(\Omega, \mathscr{B})$. The for any $t\in (0,1)$,
	\begin{equation}\nonumber
	(\text{concavity})\quad
	tH(\mu,\mathcal{Q})+(1-t)H(\mu',\mathcal{Q})\le H(t\mu+(1-t)\mu',\mathcal{Q}),
	\end{equation}
	\begin{equation}\nonumber
	\qquad\qquad
	tH(\mu,\mathcal{P}|\mathcal{Q})+(1-t)H(\mu',\mathcal{P}|\mathcal{Q})\le H(t\mu+(1-t)\mu',\mathcal{P}|\mathcal{Q}),
	\end{equation}
	\begin{equation}\nonumber
	(\text{convexity})\quad
	 H(t\mu+(1-t)\mu',\mathcal{Q})\le t H(\mu,\mathcal{Q})+(1-t)H(\mu',\mathcal{Q})+H(t).
	\end{equation}
\end{lemma}

\begin{lemma} \label{lem:affinetransform}
	Let $\mu\in \mathscr{P}(\mathbb{R}^d)$. There is a constant $C_d>0$ such that for any affine map  $f(x)=ax+c$, $a,\in \R\setminus\{0\},\, c\in\R^d$ and for any $n\in\mathbb{N}$ we have
	$$\left|H(f\mu,\,\mathcal{L}_{n+[\log_b |a|]})-H(\mu,\,\mathcal{L}_{n})\right|\le C_d.$$
\end{lemma}
%

\begin{lemma}\label{lem:pfclose}
	Given a probability space $(\Omega, \mathscr{B}, \mu)$, if $f,g:\Omega\to\mathbb{R}^d$ are measurable and $\sup_x|f(x)-g(x)|\le b^{-n}$ then
	$$\left|H(f\mu, \mathcal{L}_{n})-H(g\mu, \mathcal{L}_{n})\right|\le C_d,$$
	where $C_d$ is an absolute constant.
\end{lemma}
The following lemma is from \cite[Lemma 3.4]{hochman2014self}.
\begin{lemma}\label{lem:decomposition}
For $R\ge1$ and $\mu\in\mathscr{P}([-R,R]^2)$ and integers $m\le n$,
$$\frac1nH(\mu,\mathcal{L}_n)=\mathbb{E}^{\,\mu}_{0\le i<n}\bigg(\frac1mH(\mu^{x,\,i},\mathcal{L}_m) \bigg)+O\big(\frac mn+\frac{\log R}n\big).$$	
\end{lemma}
See (\ref{def:measureupper}) for the definition of $\mu^{x,\,i}$.

\section{ dimension conservation}\label{DC}
In this section, we shall first introduce the definition of dimension conservation. Later we will give the  dimension conservation property of measures $ m_x$.
The following notation is introduced by Furstenberg \cite{Furstenberg}.
\begin{definition}\label{def:m}
	For any $\theta\in[0,1]$, a Borel probability measure $\mu\in\mathscr{P}(\mathbb{C})$ is called  {\em $\theta$-dimension conservation}
	 if we have
	 	\begin{enumerate}
	 	\item[(1)] $\mu,\,\pi_{\theta}\mu$ are exact-dimensional;
	 	\item[(2)] for $\mu\text{-a.e.}\, z\in\mathbb{C}$,  $\mu^{\eta_{\theta}}_{z}$ is exact-dimensional and
	 	$$\text{dim}(\mu)=\text{dim}(\pi_{\theta}\mu)+\text{dim}(\mu^{\eta_{\theta}}_{z})$$
	 \end{enumerate}
 where  $\pi_{\theta}$ is from (\ref{def:pi}), 
 partition $\eta_{\theta}:=\big\{\,\pi_{\theta}^{-1}(\,\{x\}\,):\,x\in\mathbb{R}\,\big\} $ and $\{\,\mu^{\eta_{\theta}}_{z}\,\}_{z\in\mathbb{C}}$ is the canonical system of conditional measure with $\eta_{\theta}$ (See Theorem \ref{Rohlin}).
\end{definition}

  Feng and Hu \cite{Feng} studied the exact-dimensional properties of   selfsimilar measures 
 by Ledrappier-Young theory
  \cite{ledrappier1985metric}. 
 The difference is that they  consider the system in symbolic space  instead of Riemann space. They also provided some useful tools  for this approach. Based on these,  Falconer and Jin \cite{Falconer}
studied the dimension conservation properties of random self-similar. Following this strategy we have the following result, which is important to prove Theorem \ref{thm:entropygrowth}.

\begin{theorem}\label{thm:DC}
	Let $b\ge 2$ be an integer and $\gamma\in\mathbb{C}$ such that $0<|\gamma|<1$.
	Let $\phi:\mathbb{R}\to\mathbb{R}$ be a $\mathbb{Z}$-periodic Lipschitz function.
	If $\Delta\in\mathbb{R}\setminus\mathbb{Q}$, then there exists nonnegative constants $\beta,\,\upsilon$ such that the following holds.
	
	For Lebesgue-$\text{a.e.}\,(x,\theta)\in [0,1)^2,$ the measure $m_x$ is 
	$\theta$-dimension conservation. Furthermore we have
		\begin{enumerate}
		\item[(D.1)] $\text{dim}(\pi_{\theta}m_x)=\beta$;
		\item[(D.2)] for $m_x\text{-a.e.}\, z\in\mathbb{C}$,
		$\text{dim}\big(\,(m_x)^{\eta_{\theta}}_{z}\,\big)=\upsilon$.
	\end{enumerate}
\end{theorem}
Recall $\Delta=\Delta(\gamma)$ is defined by (\ref{def:Delta}).
We only consider the case $\Delta\notin\mathbb{Q}$ in the rest of paper.
The following is a  corollary of Theorem \ref{thm:DC} by the observation of Hochman \cite[Lemma 3.21(5)]{hochman2021}.  We  offer the proof for the completeness.

\begin{corollary}\label{cor:relation}
	If $\alpha\ge\beta+1$, then $\alpha\ge2.$
\end{corollary}
\begin{proof}
	By Theorem \ref{thm:DC} there exists $x,\,\theta_1,\,\theta_2\in[0,1)$ such that
	$\theta_1\neq\theta_2$ and the following holds:
	 $m_x,\,\pi_{\theta_j}m_x$ are exact-dimensional and $\text{dim}(m_x)=\alpha$, $\text{dim}(\pi_{\theta_j}m_x)=\beta$, $ j=1,2.$
	For any $n\in\mathbb{Z}_+$, since each atom of $\pi_{\theta_1}^{-1}(\mathcal{L}^{\mathbb{R}}_n)\bigvee\pi_{\theta_2}^{-1}(\mathcal{L}^{\mathbb{R}}_n)$ intersects $\boldmath{O}_{\theta_1,\,\theta_2}(1)$ atoms of $\mathcal{L}^{\mathbb{C}}_n$ and vice versa, we have
	\begin{equation}\label{lem:DC1}
	\begin{aligned}
	H(m_x,\,\mathcal{L}^{\mathbb{C}}_n)&=H\big(m_x,\,\pi_{\theta_1}^{-1}(\mathcal{L}^{\mathbb{R}}_n)\vee\pi_{\theta_2}^{-1}(\mathcal{L}^{\mathbb{R}}_n)\big)+\boldmath{O}_{\theta_1,\,\theta_2}(1)
	\\&=H(\,m_x,\,\pi_{\theta_1}^{-1}(\mathcal{L}^{\mathbb{R}}_n) )+H\big(m_x,\,\pi_{\theta_2}^{-1}(\mathcal{L}^{\mathbb{R}}_n)\mid\pi_{\theta_1}^{-1}(\mathcal{L}^{\mathbb{R}}_n)\big)+\boldmath{O}_{\theta_1,\,\theta_2}(1)
	\\&\ge H\big(m_x,\,\pi_{\theta_1}^{-1}(\mathcal{L}^{\mathbb{R}}_n)\mid\pi_{\theta_2}^{-1}(\mathcal{L}^{\mathbb{R}}_n)\big)+H\big(m_x,\,\pi_{\theta_2}^{-1}(\mathcal{L}^{\mathbb{R}}_n)\mid\pi_{\theta_1}^{-1}(\mathcal{L}^{\mathbb{R}}_n)\big)+\boldmath{O}_{\theta_1,\,\theta_2}(1).
	\end{aligned}
	\end{equation}
	The above also implies
	$$H\big(m_x,\,\pi_{\theta_1}^{-1}(\mathcal{L}^{\mathbb{R}}_n)\mid\pi_{\theta_2}^{-1}(\mathcal{L}^{\mathbb{R}}_n)\big)=H(m_x,\,\mathcal{L}^{\mathbb{C}}_n)-H(\pi_{\theta_2} m_x,\,\mathcal{L}^{\mathbb{R}}_n)+\boldmath{O}_{\theta_1,\,\theta_2}(1)$$
	and
	$$H\big(m_x,\,\pi_{\theta_2}^{-1}(\mathcal{L}^{\mathbb{R}}_n)\mid\pi_{\theta_1}^{-1}(\mathcal{L}^{\mathbb{R}}_n)\big)=H(m_x,\,\mathcal{L}^{\mathbb{C}}_n)-H(\pi_{\theta_1} m_x,\,\mathcal{L}^{\mathbb{R}}_n)+\boldmath{O}_{\theta_1,\,\theta_2}(1).$$
	Combining these with (\ref{lem:DC1}), we have 
	$$\frac1nH(m_x,\,\mathcal{L}^{\mathbb{C}}_n)\ge\sum_{k=1}^2\bigg(\frac1n H(m_x,\,\mathcal{L}^{\mathbb{C}}_n)-\frac1n H(\pi_{\theta_k} m_x,\,\mathcal{L}^{\mathbb{R}}_n)\bigg)+\boldmath{O}_{\theta_1,\,\theta_2}(\frac1n).$$
	Therefore when $n$ goes to infinity, we have
	$$\alpha\ge 2(\alpha-\beta)$$
	by Proposition \ref{prop:Young}, which implies $\alpha\ge2.$
\end{proof}
\begin{remark}
	Under the observation of Corollary \ref{cor:relation},  we only need Theorem \ref{thm:DC} (D.1) to prove Theorem \ref{thm:entropygrowth}. For completeness and the readers to understand our ideas, we still introduce  Theorem \ref{thm:DC} (D.2) and give the proof.
\end{remark} 
	
\subsection{Group extension} By the classical Ledrappier-Young theory \cite{ledrappier1992dimension,ledrappier1985metric, Shu2010}, we can only get the formula (\ref{LedrapperYoungF}). To study the dimension of $m_x$, it is important to understand the projection measures $\pi_{\theta}m_x$ and conditional measures $(m_x)^{\eta_{\theta}}_{z}$. Thus it is nature to consider the map
\begin{equation}\nonumber
\begin{aligned}
\hat{T}:\,&\Sigma\times[0,1)^2\,\to \,\Sigma\times[0,1)^2
\\&(\,\textbf{j},\,x,\,\theta)\,\mapsto\, (\,\lfloor bx\rfloor\,\textbf{j},\,bx\mod 1,\,\theta-\Delta\mod 1)
\end{aligned}
\end{equation}
where $\lfloor bx\rfloor$ is the largest integer not greater than $bx$ and see (\ref{eq:symbolsum}) for $\lfloor bx\rfloor\,\textbf{j}$. Recall $\mathcal{m}$ is the Lebesgue measure on $[0,1)$. Let 
$$\tilde{\omega}:=\nu^{\mathbb{Z}_+}\times\mathcal{m}\times\mathcal{m}$$
and $\mathcal{B}$ be the Borel set of $\Sigma\times[0,1)^2$.
For any metric space $(X,\,\rho)$, denote $\mathcal{B}_X$ as the Borel set of $X$.
 Let us consider the probability space $(\Sigma\times[0,1)^2,\,\mathcal{B},\,\tilde{\omega})$ and we have the following Lemma \ref{lem:ergodic}.
\begin{lemma}\label{lem:ergodic}
	If $\Delta\notin\mathbb{Q}$, then
  $\hat{T}$ is ergodic.
\end{lemma}
\begin{proof}
	Let $F: \varLambda^{\mathbb{Z}}\times[0,1)\to\varLambda^{\mathbb{Z}}\times[0,1)$
	be the map such that $(\textbf{j},\,\theta)\,\mapsto(\tau(\textbf{j}),\,\theta-\Delta\mod1)$
	where $\tau$ is the left shift map on $\varLambda^{\mathbb{Z}}$.
	Let us consider the probability space $(\varLambda^{\mathbb{Z}}\times[0,1),\,\mathcal{B}_{\varLambda^{\mathbb{Z}}\times[0,1)},\,\nu^{\mathbb{Z}}\times\mathcal{m})$.
   It suffice to proof that $F$ is ergodic, since $\Pi\circ F=\hat{T}\circ\Pi$ where $\Pi:\,\varLambda^{\mathbb{Z}}\times[0,1)\,\to\Sigma\times[0,1)\times [0,1)$ is the map such that
	$$(\textbf{j},\,\theta)\,\mapsto\,\bigg(j_0j_{-1}\ldots,\,\sum_{k=1}^{\infty}\frac{j_k}{b^k},\,\theta\bigg).$$
	
	Regard $[0,1) $ as a group with the addition
	$$a+b:=(a+b)\mod1\quad\quad\forall\, a,b\in [0,1 )$$
	as usual.
	Assume $F$ is not ergodic, then there exists a proper closed subgroup $H$ of $[0,1) $ and functions 
	$f'\in C(\varLambda^{\mathbb{Z}};\,H),\,h\in C(\varLambda^{\mathbb{Z}};\,[0,1) )$ such that $\Delta=h\circ\tau-f'-h$ by \cite[Theorem 5.1]{Parry}. 
	Let $1_{\infty}=\ldots111\ldots\in\Lambda^{\mathbb{Z}}$. Thus
	$f'(1_{\infty})=-\Delta\in H$, which  contradicts with that  $H$ is a proper closed subgroup of $[0,1) $.
\end{proof}
\subsection{Proof of Theorem \ref{thm:DC} (D.1)}\label{sec:d1}
In the rest of paper we denote $\log$ as $\log_b$ for convenience.
Let $R^{\phi}_{\gamma}:=\frac{2\parallel\phi\parallel_{\infty}}{1-|\gamma|}.$
Following \cite{Feng},
for any function $g:\Sigma\to\mathbb{R}^d$ and every $\textbf{j}\in\Sigma$, $n\in\mathbb{N}$, let
\begin{equation}\label{def:ball}
\mathbf{B}_{g}(\textbf{j},\,n):=g^{-1}\bigg(\,\mathbf{B}\big(g(\textbf{j}),\,R^{\phi}_{\gamma}|\gamma|^n\,\big)\,\bigg).
\end{equation}
Let $\tau$ be the left shift on $\Sigma$ as usual.
 Let $\mathcal{P}$ be the partition of $\Sigma$ such that $$\mathcal{P}=\{\,[k]\,:\,k\in\varLambda\,\}$$ where $[k]=\{\,\textbf{j}\in\Sigma:j_1=k\,\}.$
 It is necessary to recall the following basic facts about $\pi_{\theta}$. For any $\theta\in\mathbb{R}$, we have
 \begin{equation}\label{eq:projectionfunction}
 \pi_{\theta}(\gamma z)=|\gamma|\pi_{\theta-\Delta}(z)\quad\quad\forall z\in\mathbb{C}.
 \end{equation}
 Therefore (\ref{T^nm_x}) implies: for each $n\in\mathbb{Z}_+$ and $\textbf{w}\in\varLambda^n$, we have
 \begin{equation}\label{eq:projectionmeasure}
 \pi_{\theta}(T^nm_{\textbf{w}(x)})=g_{\textbf{w},\,\theta}(\pi_{\theta-n\Delta}m_{\textbf{w}(x)})
 \end{equation}
 where function $g_{\textbf{w},\,\theta}(z):=|\gamma|^nz+\pi_{\theta}\circ S_x(\textbf{w})$ for each $z\in\mathbb{R}$.
 
\begin{lemma}\label{lem:projection}
For any $\theta,x\in[0,1], $ and $\textbf{j}\in\Sigma$, $n\in\mathbb{N}$, we have
\begin{equation}\label{lem:symbal}
\mathbf{B}_{\pi_{\theta}\circ S_x}(\textbf{j},\,n+1)\cap\mathcal{P}(\textbf{j})=\tau^{-1}\bigg(\,\mathbf{B}_{\pi_{\theta-\Delta}\circ S_{(x+j_1)/b}}(\tau(\textbf{j}),\,n)\,\bigg)\cap\mathcal{P}(\textbf{j}).
\end{equation}
\end{lemma}
\begin{proof}
	For any $\textbf{i}\in\Sigma$ such that $i_1=j_1$ and
	$|\pi_{\theta}\circ S_x(\textbf{j})-\pi_{\theta}\circ S_x(\textbf{i})|\le R^{\phi}_{\gamma}|\gamma|^{n+1}$, we have
	$$\bigg|\,\pi_{\theta}\bigg(\gamma S_{(x+j_1)/b }\circ\tau(\textbf{j})- \gamma S_{(x+j_1)/b}\circ\tau(\textbf{i})\bigg)\,\bigg|\le R^{\phi}_{\gamma}|\gamma|^{n+1}$$
	by the definition of function $S_x$ and (\ref{FM}). This implies
	$$\bigg|\,\pi_{\theta-\Delta}\bigg( S_{(x+j_1)/b}\circ\tau(\textbf{j})- S_{(x+j_1)/b}\circ\tau(\textbf{i})\bigg)\,\bigg|\le R^{\phi}_{\gamma}|\gamma|^{n}$$
    by (\ref{eq:projectionfunction}). Combining the above with (\ref{def:ball}), we have
	$$\mathbf{B}_{\pi_{\theta}\circ S_x}(\textbf{j},\,n+1)\cap\mathcal{P}(\textbf{j})\subset\tau^{-1}\bigg(\mathbf{B}_{\pi_{\theta-\Delta}\circ S_{(x+j_1)/b}}(\tau(\textbf{j}),\,n)\bigg)\cap\mathcal{P}(\textbf{j}).$$
	Finally the other side can be proved by the same method.
\end{proof}
Let $\eta$  be the partition of $\Sigma\times[0,1)^2 $ such that $$\eta:=\big\{\,\Sigma\times\{x\}\times\{\theta\}:\,\theta,x\in[0,1)\, \big\}$$
and $\hat{\eta}$ be the  sub-$\sigma$-algebra of $\mathcal{B}$ generated by  $\eta$.
Let $\tilde{\mathcal{P}}$ be the sub-$\sigma$-algebra of $\mathcal{B}$ generated by partition
 $$\big\{\,[k]\times[0,1)^2:\,k\in\varLambda\,\big\}.$$
 Define the function
 $$\Phi:\Sigma\times[0,1)^2\to\mathbb{R}\quad\quad \Phi(\,\textbf{j},\,x,\,\theta)=\pi_{\theta}\circ S_x(\textbf{j}).$$
Let $\mathcal{B}_{\Phi}:=\Phi^{-1}\big(\mathcal{B}_{\mathbb{R}}\big)$ be the sub-$\sigma$-algebra of $\mathcal{B}$. The following is an immediate consequence of \cite[Proposition 3.5]{Feng}.
\begin{lemma}\label{lem:Feng0}
	For $\hat{\omega}$-a.e. $(\,\textbf{j},x,\theta)\in\Sigma\times[0,1)^2 $ we have
	\begin{equation}\label{lem:Feng}
	\lim_{n\to\infty}\log\frac{\nu^{\mathbb{Z}_+}\big(\mathbf{B}_{\pi_{\theta}\circ S_x}(\textbf{j},\,n)\cap\mathcal{P}(\textbf{j})\big)}{\nu^{\mathbb{Z}_+}\big(\mathbf{B}_{\pi_{\theta}\circ S_x}(\textbf{j},\,n)\big)}=-\mathbf{I}_{\hat{\omega}}\big(\tilde{\mathcal{P}}|\hat{\eta}\vee\mathcal{B}_{\Phi}\big)(\,\textbf{j},x,\theta).
	\end{equation}
	Furthermore, set 
	$$g(\textbf{j},x,\theta)=\sup_{n\in\mathbb{N}}-\log \frac{\nu^{\mathbb{Z}_+}\big(\mathbf{B}_{\pi_{\theta}\circ S_x}(\textbf{j},\,n)\cap\mathcal{P}(\textbf{j})\big)}{\nu^{\mathbb{Z}_+}\big(\mathbf{B}_{\pi_{\theta}\circ S_x}(\textbf{j},\,n)\big)}.$$
	Then $g\ge0$ and $g\in L^1(\Sigma\times[0,1)^2,\,\mathcal{B},\,\hat{\omega}).$
\end{lemma}
We also need the following ergodic theorem due to Maker \cite{Maker}.
\begin{theorem}\label{thm:Maker}
	Let $(\Omega,\,\mathcal{M},\,\mu,\,G)$ be a measure-preserving system and let $\{g_n\}$ be integrable functions on $(\Omega,\,\mathcal{M},\,\mu)$. If
	$g_n(x)\to g(x)$ a.e. and if $sup_n|g_n(x)|=\overline{g}(x)$ is integrable, then for a.e. $x$,
	$$\lim_{n\to\infty}\frac1n\sum_{k=0}^{n-1}g_{n-k}\circ G^k(x)=g_{\infty}(x),$$
	where $g_{\infty}(x)=\lim_{n\to\infty}\frac1n\sum_{k=0}^{n-1}g\circ G^k(x).$
\end{theorem}
For any $\textbf{j}\in\Sigma$ and $n\in\mathbb{Z}_+$, let $\textbf{j}_n=j_1j_2\ldots j_n\in\varLambda^n.$ 
\begin{lemma}\label{lem:D.1}
	For Lebesgue-$\text{a.e.}\,(x,\theta)\in [0,1)^2$ the following holds.
	 For $\pi_{\theta}m_x$-a.e. 
	$z\in\mathbb{R}$, we have
	$$\lim_{r\to 0}\frac{\log\bigg(\pi_{\theta}m_x\big(\mathbf{B}(z,\,r)\big)\bigg)}{\log r}=\frac{H_{\hat{\omega}}\big(\tilde{\mathcal{P}}|\hat{\eta}\vee\mathcal{B}_{\Phi}\big)-\log b}{\log|\gamma|}.$$
\end{lemma}
\begin{proof}
	For any $(\textbf{j},\,x,\,\theta)\in\Sigma\times[0,1)^2 $ and $n\in\mathbb{Z}_+$,
	we have
	\begin{equation}
	\begin{aligned}
		\nu^{\mathbb{Z}_+}\big(\mathbf{B}_{\pi_{\theta}\circ S_x}(\textbf{j},\,n)\big)&=	\prod_{k=0}^{n-1}\frac{\nu^{\mathbb{Z}_+}\big(\mathbf{B}_{\pi_{\theta-k\Delta}\circ S_{\textbf{j}_k(x)}}(\tau^k(\textbf{j}),\,n-k)\big)}{\nu^{\mathbb{Z}_+}\big(\mathbf{B}_{\pi_{\theta-(k+1)\Delta}\circ S_{\textbf{j}_{k+1}(x)}}(\tau^{k+1}(\textbf{j}),\,n-k-1)\big)}
		\\&=\prod_{k=0}^{n-1}\frac{\nu^{\mathbb{Z}_+}\big(\mathbf{B}_{\pi_{\theta-k\Delta}\circ S_{\textbf{j}_k(x)}}(\tau^k(\textbf{j}),\,n-k)\big)}{b\nu^{\mathbb{Z}_+}\big(\,\mathbf{B}_{\pi_{\theta-k\Delta}\circ S_{\textbf{j}_k(x)}}(\tau^k(\textbf{j}),\,n-k)\cap\mathcal{P}\big(\tau^k(\textbf{j})\big)\,\big)}
	\end{aligned}
	\end{equation}
	by the definition of $R^{\phi}_{\gamma}$ and (\ref{lem:symbal})
	where $\textbf{j}_0(x)=x$.  For any $n\in\mathbb{N}$, let function $g_n:\Sigma\times[0,1)^2\to\mathbb{Z}_+$ be such that
	$$g_n(\textbf{j},x,\theta)=-\log \frac{\nu^{\mathbb{Z}_+}\big(\mathbf{B}_{\pi_{\theta}\circ S_x}(\textbf{j},\,n)\cap\mathcal{P}(\textbf{j})\big)}{\nu^{\mathbb{Z}_+}\big(\mathbf{B}_{\pi_{\theta}\circ S_x}(\textbf{j},\,n)\big)}$$
	Thus we have
	$$\log\bigg(\nu^{\mathbb{Z}_+}\big(\,\mathbf{B}_{\pi_{\theta}\circ S_x}(\textbf{j},\,n)\,\big)\bigg)=n\log1/ b+\sum_{k=0}^{n-1}g_{n-k}\circ \hat{T}^k(\textbf{j},x,\theta),$$
	which implies: for $\hat{\omega}$-a.e. $(\textbf{j},\,x,\,\theta)\in\Sigma\times[0,1)^2 $ we have
     $$\lim_{n\to \infty}\frac{\log\bigg(\nu^{\mathbb{Z}_+}\big(\mathbf{B}_{\pi_{\theta}\circ S_x}(\textbf{j},\,n)\big)\bigg)}{n\log |\gamma|}=\frac{H_{\hat{\omega}}\big(\hat{\mathcal{P}}|\hat{\eta}\vee\mathcal{B}_{\Phi}\big)-\log b}{\log|\gamma|}$$
     by Lemma \ref{lem:Feng0} and Theorem \ref{thm:Maker}. Combining this with  $$\nu^{\mathbb{Z}_+}\big(\mathbf{B}_{\pi_{\theta}\circ S_x}(\textbf{j},\,n)\big)=\pi_{\theta}m_x\bigg(\,\mathbf{B}\big(\,\pi_{\theta}\circ S_x(\textbf{j}),\,R^{\phi}_{\gamma}|\gamma|^n\,\big)\,\bigg),$$
      thus Lemma \ref{lem:D.1} holds.
\end{proof}
\subsection{Proof of Theorem \ref{thm:DC}} This subsection is devoted to give a sketchy proof for the rest of Theorem \ref{thm:DC}, since the details are similar to Sect. \ref{sec:d1} and \cite[Sect. 3.1, Sect. 3.3]{Falconer}.  Define function $\hat{S}:\Sigma\times[0,1)^2\to\mathbb{C}$  such that
$$(\textbf{j},\,x,\,\theta)\mapsto S(x,\,\textbf{j})$$
and $\mathcal{B}_{\hat{S}}:=\hat{S}^{-1}\big(\mathcal{B}_{\mathbb{C}}\big)$. 
For any  $x,\theta\in[0,1)$, define the  partition of $\Sigma$ be such that $$\tilde{\eta}_{x,\,\theta}:=\big\{\,(\pi_{\theta}\circ S_x)^{-1}(\,\{\,z\,\}\,)\,:\,z\in\mathbb{R}\,\big\}.$$
By \cite[Proposition 3.5]{Feng} we have the following result.
\begin{lemma}
	For $\hat{\omega}$-a.e. $(\textbf{j},x,\theta)\in\Sigma\times[0,1)^2 $ we have
	\begin{equation}\label{lem:Feng2}
	\lim_{n\to\infty}\log\frac{\nu^{\mathbb{Z}_+}_{(\textbf{j},\,x,\,\theta)}\big(\mathbf{B}_{ S_x}(\textbf{j},\,n)\cap\mathcal{P}(\textbf{j})\big)}{\nu^{\mathbb{Z}_+}_{(\textbf{j},\,x,\,\theta)}\big(\mathbf{B}_{ S_x}(\textbf{j},\,n)\big)}=-\mathbf{I}_{\hat{\omega}}\big(\tilde{\mathcal{P}}|\hat{\eta}\vee\mathcal{B}_{\hat{S}}\big)(\textbf{j},x,\theta)
	\end{equation}
	where $\nu^{\mathbb{Z}_+}_{(\textbf{j},\,x,\,\theta)}=(\nu^{\mathbb{Z}_+})^{\tilde{\eta}_{x,\,\theta}}_{\textbf{j}}$ and $\big\{\,(\nu^{\mathbb{Z}_+})^{\tilde{\eta}_{x,\,\theta}}_{\textbf{j}}:\,\textbf{j}\in\Sigma\,\big\}$ is the canonical system of conditional measure with $\tilde{\eta}_{x,\,\theta}$.
	Furthermore, set 
	$$h(\textbf{j},x,\theta)=\sup_{n\in\mathbb{N}}-\log \frac{\nu^{\mathbb{Z}_+}_{(\textbf{j},\,x,\,\theta)}\big(\mathbf{B}_{ S_x}(\textbf{j},\,n)\cap\mathcal{P}(\textbf{j})\big)}{\nu^{\mathbb{Z}_+}_{(\textbf{j},\,x,\,\theta)}\big(\mathbf{B}_{ S_x}(\textbf{j},\,n)\big)}.$$
	Then $h\ge0$ and $h\in L^1(\Sigma\times[0,1)^2,\,\mathcal{B},\,\hat{\omega}).$
\end{lemma}
\begin{lemma}\label{lem:projection2}
	For any $x\in[0,1)  $ and $\textbf{j}\in\Sigma$, $n\in\mathbb{N}$, we have
	\begin{equation}\label{lem:symbal2}
	\mathbf{B}_{ S_x}(\textbf{j},\,n+1)\cap\mathcal{P}(\textbf{j})=\tau^{-1}\bigg(\mathbf{B}_{ S_{(x+j_1)/b}}(\tau(\textbf{j}),\,n)\bigg)\cap\mathcal{P}(\textbf{j}).
	\end{equation}
\end{lemma}
\begin{proof}
	The proof is  similar to Lemma \ref{lem:projection}.
\end{proof}

\begin{lemma}\label{lem: density}
	For every $x,\theta\in[0,1)$ and $A\in\mathcal{B}_{\Sigma}$, for $\nu^{\mathbb{Z}_+}$-a.e. $\textbf{j}\in\Sigma,$
	$$(\nu^{\mathbb{Z}_+})^{\tilde{\eta}_{x,\,\theta}}_{\textbf{j}}(A)=\lim_{n\to\infty}\frac{\nu^{\mathbb{Z}_+}\bigg( A\cap\mathbf{B}_{\pi_{\theta}\circ S_x}(\textbf{j},\,n)\bigg)}{\nu^{\mathbb{Z}_+}\bigg(\mathbf{B}_{\pi_{\theta}\circ S_x}(\textbf{j},\,n)\bigg)}$$
\end{lemma}
\begin{proof}
	The proof is similar to \cite[Lemma 3.1]{Falconer}.
\end{proof}

\begin{lemma}\label{lem:D.2}
	For Lebesgue-$\text{a.e.}\,(x,\theta)\in [0,1)^2$, for $m_x$-a.e. $z\in\mathbb{C}$, 
	$$\lim_{r\to 0}\frac{\log\bigg((m_x)^{\eta_{\theta}}_{z}(\mathbf{B}(z,\,r))\bigg)}{\log r}=\frac{H_{\hat{\omega}}\big(\tilde{\mathcal{P}}|\hat{\eta}\vee\mathcal{B}_{\hat{S}}\big)-H_{\hat{\omega}}\big(\tilde{\mathcal{P}}|\hat{\eta}\vee\mathcal{B}_{\Phi}\big)}{\log|\gamma|}.$$
\end{lemma}
We only give a  sketchy proof here, see \cite[Section 3.3]{Falconer} for details if the readers are unfamiliar with Ledrappier-Young theory.
\begin{proof}
	Combining Lemma \ref{lem:symbal}, Lemma \ref{lem:Feng0} with Lemma \ref{lem:projection2} and Lemma \ref{lem: density}  the following holds:
	for every  $k\in\mathbb{Z}_+$, for $\hat{\omega}$-a.e. $(\textbf{j},x,\theta)\in\Sigma\times[0,1)^2,$ we have
	$$\nu^{\mathbb{Z}_+}_{(\textbf{j},\,x,\,\theta)}\big(\mathbf{B}_{ S_x}(\textbf{j},\,k)\cap\mathcal{P}(\textbf{j})\big)=\nu^{\mathbb{Z}_+}_{\hat{T}(\textbf{j},\,x,\,\theta)}\big(\mathbf{B}_{ S_{\frac{x+j_1}b}}(\tau(\textbf{j}),\,k-1)\big)\cdot\text{exp}\bigg(-\mathbf{I}_{\hat{\omega}}\big(\tilde{\mathcal{P}}|\hat{\eta}\vee\mathcal{B}_{\Phi}\big)(\textbf{j},x,\theta)\bigg).$$
	Thus we have
	\begin{equation}\nonumber
	\begin{aligned}
	&\nu^{\mathbb{Z}_+}_{(\textbf{j},\,x,\,\theta)}\big(\mathbf{B}_{ S_x}(\textbf{j},\,k)\cap\mathcal{P}(\textbf{j})\big)\\&=
	\prod_{k=0}^{n-1}\bigg(\frac{\nu^{\mathbb{Z}_+}_{\hat{T}(\textbf{j},\,x,\,\theta)}\big(\mathbf{B}_{ S_{\textbf{j}_k(x) }
		}(\tau^k(\textbf{j}),\,n-k)\big)}{\nu^{\mathbb{Z}_+}_{\hat{T}(\textbf{j},\,x,\,\theta)}\big(\mathbf{B}_{ S_{\textbf{j}_k(x) }
		}(\tau^k(\textbf{j}),\,n-k)\cap\mathcal{P}(\tau^k(\textbf{j}))\big)}\bigg)\cdot\text{exp}\bigg(-\mathbf{I}_{\hat{\omega}}\big(\tilde{\mathcal{P}}|\hat{\eta}\vee\mathcal{B}_{\Phi}\big)\circ\hat{T}^k(\textbf{j},x,\theta)\bigg),
	\end{aligned}	
	\end{equation}
	which implies, for $\hat{\omega}$-a.e. $(\textbf{j},\,x,\,\theta)\in\Sigma\times[0,1)^2, $ we have
	$$\lim_{n\to\infty}\frac{\log\nu^{\mathbb{Z}_+}_{(\textbf{j},\,x,\,\theta)}\big(\mathbf{B}_{ S_x}(\textbf{j},\,k)\cap\mathcal{P}(\textbf{j})\big)}{n\log|\gamma|}=\frac{H_{\hat{\omega}}\big(\tilde{\mathcal{P}}|\hat{\eta}\vee\mathcal{B}_{\hat{S}}\big)-H_{\hat{\omega}}\big(\tilde{\mathcal{P}}|\hat{\eta}\vee\mathcal{B}_{\Phi}\big)}{\log|\gamma|}$$
	by Theorem \ref{thm:Maker}, Lemma \ref{lem:Feng2} and Birkhoff Ergodic Theorem.
\end{proof}
\begin{lemma}\label{lem:D.0}
	For Lebesgue-$\text{a.e.}\,(x,\theta)\in [0,1)^2$, for $m_x$-a.e. $z\in\mathbb{C}$, 
	$$\lim_{r\to 0}\frac{\log\big(m_x(\mathbf{B}(z,\,r))\big)}{\log r}=\frac{H_{\hat{\omega}}\big(\tilde{\mathcal{P}}|\hat{\eta}\vee\mathcal{B}_{\hat{S}}\big)-\log b}{\log|\gamma|}.$$
\end{lemma}
\begin{proof}
	The proof is similar to Lemma \ref{lem:D.1} by Lemma \ref{lem:projection2}.
\end{proof}
\begin{proof}[The proof of Theorem \ref{thm:DC}]
 Combining Lemma \ref{lem:D.0} and Lemma \ref{lem:D.1} with Lemma \ref{lem:D.2}, 
 we have Theorem \ref{thm:DC}  holds.
\end{proof}

\section{The inverse theorem for entropy}\label{sec:inverse}
The main result of 
this section is  Theorem \ref{thm:entropygrowth}, which is the corollary of \cite[Theorem 2.8]{hochman2021} by Theorem \ref{thm:DC}. We will  use Theorem \ref{thm:entropygrowth} to study  the convolution of $m_x$  and any  probability measures on $\mathbb{C}$ in Sect. \ref{sec:proveA}.
\begin{theorem}\label{thm:entropygrowth}
    Let  integer $b\ge 2$ and $\gamma\in\mathbb{C}$ such that $0<|\gamma|<1$.
	Let $\phi:\mathbb{R}\to\mathbb{R}$ be a $\mathbb{Z}$-periodic Lipschitz function  such that condition (H) holds. If  $\alpha<2$  and $\Delta\in\mathbb{R}\setminus\mathbb{Q}$, then the following holds.
  
  For  any $\eps>0$ and $R>0$, there are $\delta_6=\delta_6(\eps,\,R,\,\phi,\,\gamma),\,\delta_7=\delta_7(\eps,\,R,\,\phi,\,\gamma)>0$ with the following properties:
  for every $m>M(\eps,\,R,\,\phi,\,\gamma)$, $n>N(\eps,\,R,\,\phi,\,\gamma,\,m)$ and  $x\in[0,1) $, if  
  $\mu\in\mathscr{P}\big([-R,R]^2\big) $ is a measure such that
 $$\frac1n H(\mu,\,\mathcal{L}_n)>\eps,$$
 and if
 	$$\mathbb{P}^{\,m_x}_{0\le i< n} \left(\frac{1}{m} H ((m_x)^{x,\,i},\mathcal{L}^{\mathbb{C}}_{m})<\alpha+\delta_6\right) >1-\delta_6,$$
 then we have
 $$\frac1n H(\mu\ast m_x,\,\mathcal{L}_n)\ge \frac1n H(m_x,\,\mathcal{L}_n)+\delta_7.$$
\end{theorem}

\subsection{The projection theorem for entropy}
This subsection is devoted to analyze the uniform low bound of $\frac1n H(\pi_{\theta}m_x,\,\mathcal{L}_n)$ for all $x,\,\theta\in [0,1] $  when $n$ is large enough. See Lemma \ref{lem:entropyprojection} for details.
We shall first introduce the following Lemma \ref{lemNonatomic}, whose idea is inspired by \cite[Lemma 5.1]{Ren2022}.
\begin{lemma}\label{lemNonatomic}
	If condition (H) holds and $\Delta\in\mathbb{R}\setminus\mathbb{Q}$, then the measure $\pi_{\theta}m_x$ has no atom for each $x,\,\theta\in [0,1].$
\end{lemma}
\begin{proof}
	Assume the lemma fails. Since the support of $\pi_{\theta}m_x$ is contained in the  ball $\mathbf{B}(0,\frac{\parallel\phi\parallel_{\infty}}{1-|\gamma|})$ for all $x,\,\theta\in [0,1]$,
	these are $x_1,\,\theta_1\in [0,1]$ and $p\in\mathbb{R}$ such that 
	\begin{equation}\label{EqMaxActom}
	\pi_{\theta_1}m_{x_1}(\{p\})=\max_{x,\,\theta\in [0,1],\,z\in\mathbb{R}}\pi_{\theta}m_x(\{z\})>0
	\end{equation}
	by the compactness of the probability measures $\pi_{\theta}m_x$ in the weak star topology and the continuity of function $\pi_{\theta}\circ S_x$. Therefore, for any $n\in\mathbb{N}$, (\ref{FundementalFormular}) and (\ref{eq:projectionmeasure}) imply
	\begin{equation}\label{eq:x_0}
	\pi_{\theta_1}m_{x_1}(\{p\})=\frac1{b^n}\sum_{\textbf{w}\in\varLambda^n}\pi_{\theta_1-n\Delta}m_{\textbf{w}(x_1)}(\{p_{\textbf{w}}\}).
	\end{equation}
	where $p_{\textbf{w}}=\frac{p-\pi_{\theta_1}\circ S_{x_1}(\textbf{w})}{|\gamma|^n}.$
	Combining this with  (\ref{EqMaxActom}) 
	and (\ref{eq:x_0})
	we have
	$$\pi_{\theta_1-n\Delta} m_{\textbf{w}(x_1)}( \{p_{\textbf{w}}\})=	\pi_{\theta_1}m_{x_1}(\{p\})>0\quad\quad\forall\,\textbf{w}\in\varLambda^n.$$
	This implies
	\begin{equation}\label{EqActom}
	M_0:=\sup_{\textbf{w}\in\varLambda^{\#}}|P_{\textbf{w}}|<\infty,
	\end{equation}
	since the supports of the family of probability measures $\{\,\pi_{\theta}m_x\,\}_{\theta,\,x\in[0,1]}$ have uniform bound in $\mathbb{R}$.
	
	Since $S(x,0_{\infty})-S(x,1_{\infty})\not\equiv0$, there are $x_2,\,\theta_2\in [0,1) $ such that
	\begin{equation}\label{eq:notzero}
	\pi_{\theta_2}\circ S_{x_2}(0_{\infty})-\pi_{\theta_2}\circ S_{x_2}(1_{\infty})\neq 0.
	\end{equation}
	Let $n_t\in\mathbb{N}$ and $\textbf{w}^{n_t}\in\varLambda^{n_t},\,t=1,2,\ldots$ be such that $\lim_{t\to\infty}(\,\theta_1-n_t\Delta\mod1)=\theta_2$ and
	$\lim_{t\to\infty}\textbf{w}^{n_t}(x_1)=x_2$, since $\Delta\notin\mathbb{Q}$ and $x_2$ is  $1/b^{n_t}$-approximate by
	 the set $\{\,\textbf{w}(x_1):\,\textbf{w}\in\varLambda^{n_t}\}$.
	For any $t,\,m\in\mathbb{Z}_+$,
	by the definition of $p_{\textbf{w}}$ we  have
	$$\pi_{\theta_1}\circ S_{x_1}(\textbf{w}^{n_t}0_m)-\pi_{\theta_1}\circ S_{x_1}(\textbf{w}^{n_t}1_m)=|\gamma|^{n_t+m}(p_{\textbf{w}^{n_t}1_m}-p_{\textbf{w}^{n_t}0_m}),$$
	 which implies that 
	$$\big|\pi_{(\theta_1-n_t\Delta\mod1)}\circ S_{\textbf{w}^{n_t}(x_1)}(0_m)-\pi_{(\theta_1-n_t\Delta\mod1)}\circ S_{\textbf{w}^{n_t}(x_1)}(1_m)\big|\le2|\gamma|^m M_0$$
	where $1_m=11\ldots1,0_m=00\ldots0\in\varLambda^m.$ 
	This contradicts (\ref{eq:notzero})
	when $m,\,t$ goes to infinity.
\end{proof}
For any $n\in\mathbb{N}$, let $\tilde{n}$ be the unique integer such that
$b^{-\tilde{n}}\le|\gamma|^n<b^{-\tilde{n}+1}$.

\begin{lemma}\label{lem:entropyprojection}
	If the condition (H) holds and $\Delta\notin\mathbb{Q}$, we have
	$$\lim_{n\to\infty}\inf_{x,\,\theta\in[0,1] }\frac1nH(\pi_{\theta}m_x,\,\mathcal{L}_n)=\beta.$$
\end{lemma}
\begin{proof}
	 We shall first prove the following claim.
	\begin{claim}\label{claim:H}
	For any $x,\,\theta\in[0,1]$ and  $\ell,\,n\in\mathbb{Z}_+$ such that $\ell,n/\ell$ is large enough, 
	 we have
	$$\frac1nH(\pi_{\theta}m_x,\,\mathcal{L}_n)\ge\bigg(1+O\big(\frac1{\tilde{\ell}}+\frac{\tilde{\ell}}n \big)\bigg)\cdot\frac1{ t_{n,\ell} } \sum_{k=0}^{ t_{n,\ell}-1}\bigg(\frac1{b^{\tilde{\ell}}}\sum_{\textbf{w}\in\varLambda^{\tilde{\ell}}}\frac1{\tilde{\ell}}H(\pi_{\theta-k\ell\Delta}m_{\textbf{w}(0)},\,\mathcal{L}_{\tilde{\ell}})\bigg)+O(\frac1{\tilde{\ell}}+\frac{\tilde{\ell}}n)$$
	where $t_{n,\ell}=\max\{t\in\mathbb{N}:\,\widetilde{t\ell}\le n\}$.
	\end{claim}
	This is enough to conclude the proof. Indeed, for any  $\ell\in\mathbb{N}$ such that $\hat{\ell}>0$, let $f_{\ell}:[0,1 )\to\mathbb{R}$ be a function such that
	$$f_{\ell}(\theta):=\frac1{b^{\tilde{\ell}}}\sum_{\textbf{w}\in\varLambda^{\tilde{\ell}}}\frac1{\tilde{\ell}}H(\pi_{\theta}m_{\textbf{w}(0)},\,\mathcal{L}_{\tilde{\ell}}).$$
	Since $\Delta\in\mathbb{R}\setminus\mathbb{Q}$ implies $\ell\Delta\notin\mathbb{Q}$, therefore we have 
	\begin{equation}\label{eq:UniformConverge}
	\lim_{n\to\infty}\inf_{\theta\in[0,1) }\frac1n\sum_{k=0}^{n-1}f_{\ell}(\theta-k\ell\Delta)=\int_{[0,1)}f_{\ell}(s)ds
	\end{equation}
	 by the Unique Ergodic Theorem \cite[Theorem 6.19]{GTM79}.
	Also we have
	$$\int_{[0,1)}f_{\ell}(s)ds=\int_{[0,1)^2}\frac1{\tilde{\ell}}H(\pi_{\theta}m_{x},\,\mathcal{L}_{\tilde{\ell}})dxd\theta+O(\frac1{\tilde{\ell}}).$$
	Combining this with
	$$\lim_{\ell\to\infty}\int_{[0,1)^2}\frac1{\tilde{\ell}}H(\pi_{\theta}m_{x},\,\mathcal{L}_{\tilde{\ell}})dxd\theta=\beta$$
	by Proposition \ref{prop:Young}, Theorem \ref{thm:DC} and Lebesgue Convergence Theorem, we have
	$$\lim_{\ell\to\infty}\int_{[0,1)}f_{\ell}(s)ds=\beta. $$
	Combining this with (\ref{eq:UniformConverge}) and Claim \ref{claim:H}, we have Lemma \ref{lem:entropyprojection} holds.
	
	 Now let us  prove Claim \ref{claim:H}. Since $n-\widetilde{t_{n,\ell}\ell}=O_{b,\,\gamma}(\tilde{\ell})$, we have
	$$\frac1nH(\pi_{\theta}m_x,\,\mathcal{L}_n)=\frac1n\sum_{k=0}^{t_{n,\ell}-1}H(\pi_{\theta}m_x,\,\mathcal{L}_{\widetilde{(k\ell+\ell)}}|\mathcal{L}_{\widetilde{(k\ell)}})+O(\frac{\tilde{\ell}}n).$$
	By the concavity of conditional entropy, this implies
	\begin{equation}\label{eq:lowerentropy}
	\frac1nH(\pi_{\theta}m_x,\,\mathcal{L}_n)\ge\frac1n\sum_{k=0}^{t_{n,\ell}-1}\frac1{b^{k\ell}}\sum_{\textbf{w}\in\varLambda^{k\ell}}H(\pi_{\theta}(T^{k\ell} m_{\textbf{w}(x)}),\,\mathcal{L}_{\widetilde{(k\ell+\ell)}}|\mathcal{L}_{\widetilde{k\ell}})+O(\frac{\tilde{\ell}}n).
	\end{equation}
	Also  (\ref{eq:projectionmeasure}) implies the measure $\pi_{\theta}(T^{k\ell} m_{\textbf{w}(x)})$ is supported in an interval of length $O(b^{-\widetilde{k\ell}} )$.
	Thus we have
	$$H(\pi_{\theta}(T^{k\ell} m_{\textbf{w}(x)}),\,\mathcal{L}_{\widetilde{(k\ell+\ell)}}|\mathcal{L}_{\widetilde{k\ell}})=H(\pi_{\theta}(T^{k\ell} m_{\textbf{w}(x)}),\,\mathcal{L}_{\widetilde{(k\ell+\ell)}})+O(1).$$
	 Combining this with
	 $\widetilde{(k\ell+\ell)}=\widetilde{k\ell}+\tilde{\ell}+O_{b,\,\gamma}(1)$ and (\ref{eq:projectionmeasure}), we have
	 $$H(\pi_{\theta}(T^{k\ell} m_{\textbf{w}(x)}),\,\mathcal{L}_{\widetilde{(k\ell+\ell)}}|\mathcal{L}_{\widetilde{k\ell}})=H(\pi_{\theta-k\ell\Delta} m_{\textbf{w}(x)},\,\mathcal{L}_{\tilde{\ell}})+O(1).$$
	 Therefore the following holds:
	 $$	\frac1nH(\pi_{\theta}m_x,\,\mathcal{L}_n)\ge\frac1n\sum_{k=0}^{t_{n,\ell}-1}\frac1{b^{k\ell}}\sum_{\textbf{w}\in\varLambda^{k\ell}}H(\pi_{\theta-k\ell\Delta}m_{\textbf{w}(x)},\,\mathcal{L}_{\tilde{\ell}})+O(\frac{\tilde{\ell}}n+\frac{t_{n,\ell}}n).$$
	 Since $\parallel\pi_{\theta-k\ell\Delta}\circ S_{\textbf{w}(x)}-\pi_{\theta-k\ell\Delta}\circ S_{(\tau^{k\ell-\tilde{\ell}}\textbf{w})(0)}\parallel_{\infty}=O(b^{-\tilde{\ell}})$
	for any $k\ge L_0(b,\gamma),$ therefore we have
	$$\frac1nH(\pi_{\theta}m_x,\,\mathcal{L}_n)\ge\frac1n\sum_{k=0}^{t_{n,\ell}-1}\frac1{b^{\tilde{\ell}}}\sum_{\textbf{w}\in\varLambda^{\tilde{\ell}}}H(\pi_{\theta-k\ell\Delta}m_{\textbf{w}(0)},\,\mathcal{L}_{\tilde{\ell}})+O(\frac{\tilde{\ell}}n+\frac{t_{n,\ell}}n).$$
	Combining this with $O(\frac{t_{n,\ell}}n)=O(\frac1{\tilde{\ell}})$ and
	
	 $$\frac1n=\frac{t_{n,\ell}\tilde{\ell}}n\cdot\frac1{t_{n,\ell}}\cdot\frac1{\tilde{\ell} }=\bigg(1+0(\frac1{\tilde{\ell}}+\frac{\tilde{\ell}}n)\bigg)\cdot\frac1{t_{n,\ell}}\frac1{\tilde{\ell} },$$
	 we have this claim holds.
\end{proof}

\subsection{An inverse theorem for entropy}
In this section, we shall introduce the inverse theorem for convolutions on $\mathbb{R}^d$  due to Hochman \cite{hochman2021}. Before that it is necessary to recall the following notations from  \cite{hochman2021}.

For a linear subspace $V\le\mathbb{C}$, let $W=V^{\perp}$ be the orthogonal complement of $V$ and let $\pi_V$ be the orthogonal projection on $V$. Thus, if $V=\{\,te^{2\pi i\theta}\,\}_{t\in\mathbb{R} }$ for some $\theta\in\mathbb{R}$, then we have $\pi_{V}=\pi_{\theta}.$ For any set $A\subset\mathbb{C}$ and $\eps>0$  write $\eps$-neighborhood of $A$ by
$$A^{(\eps)}:=\{z\in\mathbb{C}:\,d(z,A)<\eps\}.$$
\begin{definition}
	Let $V\le\mathbb{C}$ be a linear subspace and $\eps>0$. A measure $\mu\in\mathscr{P}(\mathbb{C})$ is $(V,\eps)$-concentrated if there is a translate $W$ such that $\mu(W^{(\eps)})\ge1-\eps.$
\end{definition}
We also need the following definition.
\begin{definition}
	Let $V\le\mathbb{C}$ be a linear subspace, $W=V^{\perp}$ is the orthogonal complement, and $\eps>0$. A probability measure $\mu\in\mathscr{P}(\mathbb{C})$ is $(V,\eps)$-saturated at scale $m$, or $(V,\eps, m)$-saturated, if
	$$\frac1mH(\mu,\mathcal{L}_m)\ge\frac1mH(\pi_{W}\mu,\mathcal{L}_m)+\text{dim}(V)-\eps.$$
\end{definition}
\begin{remark}\label{rem:saturated}
	It is worth noting that for the special case $V=\mathbb{C}$, $(V,\eps, m)$-saturated implies 
	$\frac1mH(\mu,\mathcal{L}_m)\ge 2-\eps$. Also $(V,\eps)$-saturated at scale $m$ is trivial if $V=\{0\}$.
\end{remark}
 The following Theorem \ref{thm:hochman} is the inverse theorem of Hochman, which serves as a basic tool to study the dimension of $m_x$ in our paper. See \cite[Theorem 2.8]{hochman2021}.
\begin{theorem}[Hochman]\label{thm:hochman}
	For any $R,\eps>0$ and $m\in\mathbb{N}$ there is a $\delta_0=\delta_0(\eps,\,R,\,m)>0$ such that for every $n>N_0(\eps,\,R,\,\delta_0,\,m)$, the following holds: if $\mu,\zeta\in\mathscr{P}([-R,R]^2)$ and
	$$\frac1nH(\mu\ast\zeta,\mathcal{L}_n)<\frac1nH(\mu,\mathcal{L}_n)+\delta_0,$$
	then there exists a sequence $V_0,\ldots V_n\le\mathbb{C}$ of subspace such that
	\begin{equation}\label{eq:hochman}
	\mathbb{P}^{\,\mu\times\zeta}_{0\le i< n} \left(\,
	\begin{matrix}
	\mu^{x,\,i}\,\text{is}\, (V_i,\,\eps,\,m)-\text{saturated and} \\
	\zeta^{y,\,i}\,\text{is}\, (V_i,\,\eps)-\text{concentrated}
	\end{matrix}
	\,\right) > 1-\varepsilon.
	\end{equation}
\end{theorem}
Note that the definitions of $\mu^{x,\,i},\,\zeta^{y,\,i}$ are from (\ref{def:measureupper}).
\subsection{The proof of Theorem \ref{thm:entropygrowth}}
We shall first introduce the following Lemma \ref{lem:projectionlow}, which is an important observation for proving Theorem \ref{thm:entropygrowth}.
\begin{lemma}\label{lem:projectionlow}
	Assume the condition (H) holds and $\Delta\in\mathbb{R}\setminus\mathbb{Q}.$
	Then for any $\delta,\,\eps>0$, $m\ge M_2(\eps,\delta),$ and $k\ge K_1(\eps,\delta,m)$, the following holds:
	\begin{equation}
	\inf_{x,\,\theta\in[0,1]}\mathbb{P}^{\,m_x}_{i=k}\bigg(\,\frac1mH\big(\pi_{\theta}(m_x)^{z,\,i},\mathcal{L}_m \big)\ge\beta-\eps\,\bigg)\ge1-\delta.
	\end{equation}
\end{lemma}
Before giving the proof, we shall introduce the following results \cite[Lemma 3.3]{barany2019hausdorff} and \cite[Lemma 3.4]{hochman2021}.
 Recall that the total variation distance between $\mu,\eta\in\mathscr{P}(\mathbb{C})$
is 
$$\parallel\mu-\eta\parallel=\sup_{A\in\mathscr{B}(\mathbb{C})}|\,\mu(A)-\eta(A)|.$$
\begin{lemma}\label{lem:hochman1}
	For every $\eps>0$ there exists a $\delta_1=\delta_1(\eps)>0$ with the following property. For any $0<\delta\le\delta_1,$
	suppose that a probability  measure $\eta\in\mathscr{P}(\mathbb{C})$ can be written as a convex combination $\eta=(1-\delta)\eta'+\delta\eta''$. Then for every $k$, we have
	$$\mathbb{P}^{\,\eta}_{i=k}\big(z\,:\,\parallel\eta_{z,\,i}-\eta'_{z,\,i}\parallel<\eps \big)>1-\eps.$$
\end{lemma}
In fact  \cite[Lemma 3.3]{barany2019hausdorff} only dealt with the case $\eta\in\mathscr{P}(\mathbb{R})$, however the proof  also works for all $\eta\in\mathscr{P}(\mathbb{R}^d)$.
\begin{lemma}\label{lem:hochman2}
	If $\mathcal{A}$ is the partition of $\mathbb{R}^d$, and if $\mu,\eta\in\mathscr{P}(\mathbb{R}^d)$ are supported at most $k$ atoms of partition $\mathcal{A}$  and $\parallel\mu-\eta\parallel<\eps,$ then
	$$|H(\mu,\,\mathcal{A})-H(\eta,\,\mathcal{A})|<2\eps\log_b k+2H(\frac12\eps)$$
	where  $H(\frac12\eps)$ is defined by (\ref{def:H}).
\end{lemma}
{\bf Notation.}  For any $n\in\mathbb{N}$, let $\hat{n}$ be the unique integer such that
\begin{equation}\label{eq:hat}
|\gamma|^{\hat{n}}\le b^{-n} <|\gamma|^{\hat{n}-1}
\end{equation}
as \cite{Ren2022} did.
We also need the following corollary of Lemma \ref{lemNonatomic}.
\begin{corollary}\label{cor:uppermeasurebound}
     If  the condition (H) holds, then for any $\eps>0$, there is a $\delta_2=\delta_2(\eps)>0$ such that the following holds. For any $x\in[0,1]$ and $n\in\mathbb{N}$, we have
     \begin{equation}\label{eq:uppermeasurebound}
     m_x\big(\partial\mathcal{L}_n^{(\delta_2/b^{n})}\big)<\eps
     \end{equation}
     where set $\partial\mathcal{L}_n:=\bigcup_{k_1,k_2\in\mathbb{Z}}\bigg\{z\in\mathbb{C}:\,\text{Re}(z)=\frac{k_1}{b^n}\,\text{or}\,\text{Im}(z)=\frac{k_2}{b^n}\bigg\}.$
\end{corollary}
\begin{proof}
	By Lemma \ref{lemNonatomic} there is $\delta'>0$ such that 
	\begin{equation}\label{eq:upperbound}
	\sup_{x,\,\theta\in[0,1],\,z\in\mathbb{R}}\pi_{\theta}m_x\big(\mathbf{B}(z,\delta')\big)\le\frac{\eps}2,
	\end{equation}
	since the probability measures $\pi_{\theta}m_x$ are compact in the weak star topology and the functions $\pi_{\theta}\circ S_x(\textbf{j})$ are continuous.
	Recall $R^{\phi}_{\gamma}=\frac{2\parallel\phi\parallel_{\infty}}{1-|\gamma|}.$ There is constant $\ell\in\mathbb{N}$ and $c_{\ell}>0$ such that
	\begin{equation}\label{eq:choose}
	c_{\ell}>\frac1{b^n|\gamma|^{\hat{n}+\ell}}>R^{\phi}_{\gamma}\quad\quad\forall n\in\mathbb{N}.
	\end{equation}
	For any set $A\subset\mathbb{C}$ and $z_1,\,z_2\in\mathbb{C}$, denote the set
	$$z_1A+z_2=\big\{\,z_1\cdot z+z_2\,:\,z\in A\,\}.$$
	Let $\delta_2=\delta'/c_{\ell}$.  For any $n\in\mathbb{N}$, combining (\ref{FundementalFormular}) with (\ref{T^nm_x}) we have
	\begin{equation}\label{eq:setboundary}
	m_x\big(\partial\mathcal{L}_n^{(\delta_2/b^{n})}\big)=\frac1{b^{\hat{n}+\ell} }\sum_{\textbf{w}\in\varLambda^{\hat{n}+\ell}}m_{\textbf{w}(x) }\bigg(\frac{\partial\mathcal{L}_n^{(\delta_2/b^{n})}-S(x,\textbf{w})}{\gamma^{\hat{n}+\ell}}\bigg).
	\end{equation}
	Also the definition of $\partial\mathcal{L}_n$ and (\ref{eq:choose}) imply that: for each $\textbf{w}\in\varLambda^{\hat{n}+\ell}$ such that
	$$m_{\textbf{w}(x) }\bigg(\frac{\partial\mathcal{L}_n^{(\delta_2/b^{n})}-S(x,\textbf{w})}{\gamma^{\hat{n}+\ell}}\bigg)>0,$$	 there are $k_1^{\textbf{w}},\,k_2^{\textbf{w}}\in\mathbb{Z}$ such that
	$$\bigg(\frac{\partial\mathcal{L}_n-S(x,\textbf{w})}{\gamma^{\hat{n}+\ell}}\bigg)\bigcap\text{supp}\big(m_{\textbf{w}(x)}\big)\subset \frac{\big\{z\in\mathbb{C}:\,\text{Re}(z)=
		\frac{k_1^{\textbf{w}}}{b^n}\,\text{or}\,\text{Im}(z)=\frac{k_2^{\textbf{w}}}{b^n}\big\}-S(x,\textbf{w})}{\gamma^{\hat{n}+\ell}},$$
	which implies
	$$\bigg(\frac{\partial\mathcal{L}_n^{(\delta_2/b^{n})}-S(x,\textbf{w})}{\gamma^{\hat{n}+\ell}}\bigg)\bigcap\text{supp}\big(m_{\textbf{w}(x)}\big)\subset\pi_{\theta_{\textbf{w}}}^{-1}\bigg(\mathbf{B}(z_1^{\textbf{w}},\frac{\delta_2}{b^n|\gamma|^{\hat{n}+\ell}})\bigg)\bigcup\pi_{\theta'_{\textbf{w}}}^{-1}\bigg(\mathbf{B}(z_2^{\textbf{w}},\frac{\delta_2}{b^n|\gamma|^{\hat{n}+\ell}})\bigg)$$
	for some $\theta_{\textbf{w}},\,\theta'_{\textbf{w}}\in[0,1]$ and $z_1^{\textbf{w}},\,z_2^{\textbf{w}}\in\mathbb{R}.$ Therefore we have
	$$m_x\big(\partial\mathcal{L}_n^{(\delta_2/b^{n})}\big)<\eps$$
	by (\ref{eq:upperbound}), (\ref{eq:choose}) and (\ref{eq:setboundary}), $\delta_2=\delta'/c_{\ell}$.
\end{proof}

\begin{proof}[The proof of Lemma \ref{lem:projectionlow}]
For any $\delta,\,\eps>0$, choose $\delta_1=\delta_1\big( \min\{\delta,\,\eps/8\}\big)\in(0,1)$ by Lemma \ref{lem:hochman1}. Thus there is a $\delta_2=\delta_2(\delta_1)>0$	such that (\ref{eq:uppermeasurebound}) holds by Corollary \ref{cor:uppermeasurebound}. Let $\ell\in\mathbb{N}$ be a constant such that 
\begin{equation}\label{eq:uppercountrol}
2R^{\phi}_{\gamma}|\gamma|^{\hat{n}+\ell}\le\frac{\delta_2}{b^n}\quad\quad\forall n\in\mathbb{N}.
\end{equation}
By Lemma \ref{lem:entropyprojection} there is a number $M_1=M_1(\eps)\in\mathbb{N}$ such that, for all $m\ge M_1$ we have
\begin{equation}
\inf_{x,\,\theta\in[0,1] }\frac1mH(\pi_{\theta}m_x,\,\mathcal{L}_m)\ge\beta-\eps/3.
\end{equation}

For any $x,\,\theta\in[0,1]$ and  $k\ge M_1$, let $A_{k,\,\ell}\subset\varLambda^{\hat{k}+\ell}$ be the union of all elements $\textbf{w}\in\varLambda^{\hat{k}+\ell}$ such that $\text{supp}(T^{\hat{k}+\ell}m_{\textbf{w}(x)})\subset I_{\textbf{w}}$ for some $I_{\textbf{w}}\in\mathcal{L}_{k}^{\mathbb{C}}$. Let $A^c_{k,\,\ell}:=\varLambda^{\hat{k}+\ell}\setminus A_{k,\,\ell}.$ Thus for each $\textbf{w}\in A^c_{k,\,\ell}$
we have $\text{supp}(T^{\hat{k}+\ell}m_{\textbf{w}(x)})\subset \partial\mathcal{L}_k^{(\delta_2/b^{k})}$ by (\ref{eq:uppercountrol}), therefore we have
	\begin{equation}
	\frac{\#A^c_{k,\,\ell}}{b^{\hat{k}+\ell}}\le\delta_1
	\end{equation}
by Corollary \ref{cor:uppermeasurebound}. Note that $\delta_1<1$ implies $A_k\neq\emptyset $.
 Define $$\eta':=\frac1{\#A_{k,\,\ell}}\sum_{\textbf{w}\in A_{k,\,\ell}}T^{\hat{k}+\ell}m_{\textbf{w}(x)}$$ and $\eta''\in\mathscr{P}(\mathbb{C})$ such that
	$$m_x=\bigg(1-\frac{\#A^c_{k,\,\ell}}{b^{\hat{k}+\ell}}\bigg)\,\eta'+\bigg(\frac{\#A^c_{k,\,\ell}}{b^{\hat{k}+\ell}}\bigg)\, \eta''.$$
Combining this with Lemma \ref{lem:hochman1}, we have
$$\mathbb{P}_{i=k}^{m_x}\bigg(z\,:\,\parallel(m_x)_{z,\,i}-\eta'_{z,\,i}\parallel<\eps/8 \bigg)>1-\delta.$$
Thus Lemma \ref{lem:hochman2} implies that the following holds
\begin{equation}\label{eq:probability}
\mathbb{P}_{i=k}^{m_x}\bigg(z\,:\,\big|\frac1mH\big((m_x)_{z,\,i},\pi_{\theta}^{-1}(\mathcal{L}_{i+m})\big)-\frac1mH\big(\eta'_{z,\,i},\pi_{\theta}^{-1}(\mathcal{L}_{i+m})\big)  \big|\le\frac14\eps+\frac{2H(\eps/8)}m\,  \bigg)>1-\delta.
\end{equation}
For any $I\in\mathcal{L}_{k}$, define the set $B_{I,\,k}:=\big\{\,\textbf{w}\in A_{k,\,\ell}:\,\text{supp}(T^{\hat{k}+\ell}m_{\textbf{w}(x)})\subset I\,\big\}$. For the case $B_{I,k}\ne\emptyset$, we have 
\begin{equation}\label{eq:low}
\begin{aligned}
\frac1mH\big(\eta'_I,\pi_{\theta}^{-1}(\mathcal{L}_{k+m})\big)&\ge\frac1{\#B_{I,\,k}}\sum_{\textbf{w}\in B_{I,\,k}}\frac1mH\big(T^{\hat{k}+\ell}m_{\textbf{w}(x)} ,\pi_{\theta}^{-1}(\mathcal{L}_{k+m})\big)
\\&=\frac1{\#B_{I,\,k}}\sum_{\textbf{w}\in B_{I,\,k}}\frac1mH\big(\pi_{\theta-(\hat{k}+\ell)\Delta}m_{\textbf{w}(x)} ,\mathcal{L}_{m}\big)+O(\frac{\ell}m)
\\&\ge\beta-\eps/3+O(\frac{\ell}m)
\end{aligned}
\end{equation}
by the concavity of conditional entropy and (\ref{T^nm_x}), (\ref{eq:uppercountrol}). Also we have
\begin{equation}\label{eq:eq}
\frac1mH\big((m_x)_{z,\,k},\pi_{\theta}^{-1}(\mathcal{L}_{k+m})\big)=\frac1mH\big(\pi_{\theta}(m_x)^{z,\,k},\mathcal{L}_{m}\big)+O(\frac1m).
\end{equation}
Let $M_3\in\mathbb{N}$ be large enough such that 
$O(\frac1{m})+O(\frac{\ell}{m})+\frac{2H(\eps/8)}{m}\le\eps/6$ for all $m\ge M_3$. 
Thus for $m\ge M_2:=\max\{M_1,M_3\}$, we have
$$\mathbb{P}_{i=k}^{m_x}\bigg(z\,:\,\frac1mH\big(\pi_{\theta}(m_x)^{z,\,k},\mathcal{L}_{m}\big)\ge\beta-\eps\bigg)>1-\delta$$
by (\ref{eq:probability}), (\ref{eq:low}) and (\ref{eq:eq}).
\end{proof}
We  still need the following observation from \cite{hochman2014self}.
\begin{lemma}\label{lem:Hochman3}
	For any $\eps>0$ and $R>0$, there is $\delta_3=\delta_3(\eps, R)>0$ such that the following holds.
	For all $n\ge N_1(\eps,\,\delta_3, R)$ and $\zeta\in\mathscr{P}([-R,R]^2)$, if 
	\begin{equation}\nonumber
	\mathbb{P}^{\,\zeta}_{0\le i< n} \left(\,
	\begin{matrix}
	\zeta^{x,\,i}\,\text{is}\, (0,\,\delta_3)-\text{concentrated}
	\end{matrix}
	\,\right) > 1-2\delta_3,
	\end{equation}
	then
	$$\frac1nH(\zeta,\mathcal{L_n})<\eps.$$
\end{lemma}
\begin{proof}
Without loss generality we may assume $\delta_3=1/b^m$ for some $m\in\mathbb{Z}_+$.
Since by Lemma \ref{lem:decomposition} we have	
$$\frac1nH(\zeta,\mathcal{L}_n)=\mathbb{E}^{\,\zeta}_{0\le i<n}\bigg(\frac1mH(\zeta^{x,i},\mathcal{L}_m) \bigg)+O\big(\frac mn+\frac{\log R}n\big),$$		
thus it  suffice to prove the following claim.	
\begin{claim}
	For any $\eps>0$,  $m\ge M(\eps)$ and measure $\mu\in\mathscr{P}([0,1]^2)$, the following holds.
	If $\mu\,\text{is}\, (0,\,1/b^m)-\text{concentrated}$, then
	$$\frac1mH(\mu,\mathcal{L}_m)\le\eps/2.$$
\end{claim}
Indeed we only need to choose $n$  large enough. In the rest we shall prove the claim. By the definition of $(0,\,1/b^m)$-concentrated,
there are $\tau',\tau''\in\mathscr{P}([0,1]^2)$ such that $\text{supp}(\tau')\subset\boldsymbol{B}(p,1/b^m)$ for some $p\in\mathbb{C}$ and
$$\mu=(1-t)\tau'+t\tau''$$
where $t\le1/b^m.$ By the convexity of entropy we have
$$\frac1mH(\mu,\mathcal{L}_m)\le(1-t)\,\frac1mH(\tau',\mathcal{L}_m)+t\,\frac1mH(\tau'',\mathcal{L}_m)+H(t)=O(\frac1m)$$
since $\tau'$ intersects at most nine elements of $\mathcal{L}_m$ and $H(t)=O(\sqrt{t})$.
\end{proof}
\begin{proof}[The proof of Theorem \ref{thm:entropygrowth}] Choose $\delta_3=\delta_3(\eps,R)$ and $N_1(\eps,\delta_3,R)$ by Lemma \ref{lem:Hochman3}. Also $\alpha<2$ implies
that $\alpha<\beta+1$ by Corollary \ref{cor:relation}.  let $$\delta_6:=(\frac19\min\{\delta_3,\beta+1-\alpha,\,2-\alpha,\,1\})^2.$$
Also by Lemma \ref{lem:projectionlow}, for any $m\ge M_2(\delta_6,\delta_6)>0$ and $k\ge K_1(\delta_6,\delta_6,m)$ such that
\begin{equation}\label{eq:lowprobabilty}
\inf_{x,\,\theta\in[0,1]}\mathbb{P}^{\,m_x}_{i=k}\bigg(\,\frac1mH\big(\pi_{\theta}(m_x)^{z,\,i},\mathcal{L}_m \big)\ge\beta-\delta_6\,\bigg)\ge1-\delta_6.
\end{equation}	
By Theorem \ref{thm:hochman} choose $\delta_7=\delta_0(\delta_6,R,m)$ and $N_0(\delta_6,R,\delta_7,m)$. Let $M= M_2(\delta_6,\delta_6)$
and
$N=N(\eps,\phi,\gamma)=\max\{N_0,N_1, K_1/\delta_3\}$.

 If Theorem \ref{thm:entropygrowth} fails for some $m\ge M$ and $n\ge N$, then there exists a sequence $V_0,\ldots V_{n-1}\le\mathbb{C}$ of subspace such that
(\ref{eq:hochman}) holds for measures $m_x $ and $\mu$. Thus, by Markov equality there are subset $I\subset\{0,1,\ldots,n-1\}$ such that $\frac{\# I}b\ge1-\sqrt{\delta_6}$ and
\begin{equation}\label{eq:hochman2}
\mathbb{P}^{\,m_x\times\mu}_{i=k} \left(\,
\begin{matrix}
(m_x)^{z,\,i}\,\text{is}\, (V_i,\,\delta_6,\,m)-\text{saturated and} \\
\mu^{y,\,i}\,\text{is}\, (V_i,\,\delta_6)-\text{concentrated}
\end{matrix}
\,\right) \ge1-\sqrt{\delta_6}
\end{equation}
for each $k\in I.$ Let $I_j$ be the union of all the elements $t\in I$ such that $\text{dim}(V_t)=j$ for $j=0,1,2$.
Also our condition implies that
there are subset $J\subset\{0,1,\ldots,n-1\}$ such that $\frac{\# J}b\ge1-\sqrt{\delta_6}$ and
\begin{equation}\label{eq:hochman3}
	\mathbb{P}^{\,m_x}_{i=k} \left(\frac{1}{m} H ((m_x)^{x,\,i},\mathcal{L}^{\mathbb{C}}_{i+m})<\alpha+\delta_6\right) \ge1-\sqrt{\delta_6}
\end{equation}
for each $k\in J.$ Since $\alpha+\delta_6<2-\delta_6$ and $1-\sqrt{\delta_6}>1/2$, we have $I_2\cap J=\emptyset$ by Remark \ref{rem:saturated}, (\ref{eq:hochman2}) and (\ref{eq:hochman3}). Therefore
we have
\begin{equation}\label{eq:number}
\frac{\#I_2}n\le\sqrt{\delta_6}.
\end{equation}
Also we have $\alpha+\delta_6<\beta+1-\delta_6$. Thus (\ref{eq:lowprobabilty}), (\ref{eq:hochman2}), (\ref{eq:hochman3}) and $\frac{K_1}n\le\delta_6$ imply that
\begin{equation}\nonumber
\frac{\#I_1}n\le2\sqrt{\delta_6}.
\end{equation}
Combining this with (\ref{eq:number}) we have
$$\frac{\#I_0}n\ge1-4\sqrt{\delta_6}.$$
Thus we have
\begin{equation}\nonumber
\mathbb{P}^{\,\mu}_{0\le i<n} \left(\,
\begin{matrix}
\mu^{y,\,i}\,\text{is}\, (0,\,\delta_6)-\text{concentrated}
\end{matrix}
\,\right)\ge (1-4\sqrt{\delta_6})(1-\sqrt{\delta_6}) \ge1-\delta_3
\end{equation}
by (\ref{eq:hochman2}).
This implies
$\frac1nH(\mu,\mathcal{L}_n)<\eps$ by Lemma \ref{lem:Hochman3}, which contradicts our condition.
\end{proof}

\section{The proof of Theorem A}\label{sec:proveA}
In this section,  we will use Theorem \ref{thm:entropygrowth} to finish the proof of  Theorem A.  Since the  proof is similar to \cite{Ren2022},  some details will be omitted. In the rest of paper, we suppose that 
$\phi(x)$ is a real analytic $\mathbb{Z}$-periodic function such that the condition (H) holds for fixed integer $b\ge2$ and $\gamma\in\mathbb{C}$ such that $0<|\gamma|<1$. 
\subsection{Entropy Porosity} In this subsection, we first recall the definition of entropy porosity from \cite{barany2019hausdorff}, then we will show the entropy porosity properties of $m_x$ for all $x\in[0,1]$.
\begin{definition}[Entropy porous]
	A measure $\mu \in \mathscr{P}(\mathbb{C})$ is $(h,\delta,m)$-\em{ entropy porous} from scale $n_1$ to $n_2$ if
	$$\mathbb{P}^{\,\mu}_{n_1\le i < n_2} \left(\frac{1}{m} H (\mu^{x,\,i},\mathcal{L}^{\mathbb{C}}_{m})<h+\delta\right) >1-\delta.$$
\end{definition}

The main result of this subsection is the following Theorem~\ref{thm:entporous}. The idea of the proof is inspired by \cite[Proposition 3.2]{barany2019hausdorff} and \cite[Theorem 5.1]{Ren2022}.
\begin{theorem}\label{thm:entporous}
	Assume the condition (H) holds and $\Delta\in\mathbb{R}\setminus\mathbb{Q}$. Then
	for any $\eps>0$, $m\ge M_4(\eps),$ $k\ge K_4(\eps,m)$ and $n\ge N_4(\eps,m,k)$, the following holds:	
	$$
	\nu^{n}\left(\left\{\textbf{i}\in \varLambda^{n}: m_{\textbf{i}(0)} \text{ is } (\alpha, \eps, m)-\text{entropy porous from scale } 1 \text{ to } k\right\}\right)>1-\eps.
	$$
\end{theorem}
Before giving the proof of Theorem \ref{thm:entporous}, we shall introduce the following Lemma \ref{lem:boundsinm}. The proof is the  same as \cite[Lemma 5.2]{Ren2022} by Lemma \ref{lemNonatomic}, thus we omit the details.
\begin{lemma}\label{lem:boundsinm}
	For any $\varepsilon>0, m\ge M_5(\varepsilon), n\ge N_5(\varepsilon,m)$,
	$$\inf\limits_{x\in [0,1]}\mathbb{\nu}^{n}\left(\,\left\{\textbf{i}\in \varLambda^{n}: \alpha-\varepsilon<\frac{1}{m}
	H(m_{\textbf{i}(x)}, \mathcal{L}_{m})<\alpha+\varepsilon\right\}\,\right) >1-\varepsilon.$$
\end{lemma}
The following is a corollary of Lemma \ref{lem:boundsinm}.
\begin{corollary}\label{lem:projectionlow2}
	Assume the condition (H) holds and $\Delta\in\mathbb{R}\setminus\mathbb{Q}.$
	Then for any $\eps>0$, $m\ge M_6(\eps)$ and $k\ge K_6(\eps,m)$, the following holds:
	\begin{equation}
	\inf_{x\in[0,1]}\mathbb{P}^{\,m_x}_{i=k}\bigg(\,\frac1mH\big((m_x)^{z,\,i},\mathcal{L}_m \big)\ge\alpha-\eps\,\bigg)\ge1-2\eps.
	\end{equation}
\end{corollary}
\begin{proof}For any $k,\,\ell\in\mathbb{Z}_+$ large enough,
   define the set 
   $$B_{k,\,\ell}=B_{k,\,\ell}(\eps,m):=\left\{\textbf{i}\in \varLambda^{\hat{k}+\ell}: \frac{1}{m}
   H(m_{\textbf{i}(x)}, \mathcal{L}_{m})>\alpha-\varepsilon/8\right\}$$
and $\tilde{A}_{k,\,\ell}:=A_{k,\,\ell}\cap B_{k,\,\ell}$
 where $A_{k,\,\ell}$ is from Lemma \ref{lem:projectionlow}. let $\tilde{A}^c_{k,\,\ell}:= \varLambda^{\hat{k}+\ell}\setminus\tilde{A}_{k,\,\ell}$, thus
 Corollary \ref{lem:projectionlow2} can be proved by the similar method of Lemma \ref{lem:projectionlow} with Lemma \ref{lem:boundsinm}.
\end{proof}
\begin{lemma}\label{lem:entporous3}
	Assume the condition (H) holds and $\Delta\in\mathbb{R}\setminus\mathbb{Q}.$
	For any $\varepsilon>0$, there exists $\delta>0$ such that if 	
	$m\ge M_7(\varepsilon)$ and $k\ge K_7(\varepsilon,m)$ and if $\left|\frac{1}{k} H(m_x, \mathcal{L}_k)-\alpha\right|<\frac{\delta}{2}$, then $m_x$ is $(\alpha, \varepsilon, m)$-entropy porous from scale $1$ to $k$.
\end{lemma}
\begin{proof}
	The method of  proof is similar to \cite[Lemma 3.7]{barany2019hausdorff} by Lemma \ref{lem:decomposition} and Corollary \ref{lem:projectionlow2}, thus we omit the details.
\end{proof}
\begin{proof}[The proof of Theorem \ref{thm:entporous}]
	The proof is  same  as \cite[Theorem 5.1]{Ren2022}.
\end{proof}

\subsection{Exponential separation}\label{sep}
In this subsection, we deduce from the condition (H) the  exponential separation properties as \cite{Ren2022} did.
\begin{definition}\label{def:ex}
	Let $E_1,E_2,\ldots$ be  subsets of $\mathbb{C}$. 
	For any $\eps>0$ and  $Q\subset\mathbb{Z}_+$,
	we say that the sequence $(E_n)_{n\in\mathbb{Z}_+}$ is $(\varepsilon,Q)$ -exponential separation if
	$$|p-q|>\sqrt{2}\,\varepsilon^{\hat{n} }\quad\quad \forall p\neq q\in E_{\hat{n}}$$
	for each $n\in Q.$
\end{definition}
The following is a corollary of \cite[Lemma 5.8]{hochman2014self}.
\begin{corollary}\label{lemHochmanSeperation2}
	For any $k\in\mathbb{N}$ and compact interval $J\subset\mathbb{R}$,
	let $F:J\to\mathbb{C}$ be a $k+1$-times continuously differentiable function. Let $M=\parallel F\parallel_{J,\,k+1}$, and let $0<d<1$ be such that for every $x\in J$ there is a $p\in\{0,1,\ldots,k\}$ with $|F^{(p)}(x)|>d.$ Then for every $0<\rho<(d/2)^{2^k}$, the set $F^{-1}\big(\mathbf{B}(0,\rho)\big)\subset J$ can be cover by $O_{k, M,|J|}(1/{d^{k+1}})$ intervals of length $\le2(\rho/d)^{1/{2^k}}$ each.
\end{corollary}
\begin{proof}
	Let $J_t=[a_t, b_t),$ $t=1,2,\ldots,L$ be the disjoint intervals such that $J=\bigcup_{t=1}^{L}J_t$ and $|J_t|<\frac{d}{4M}$ where $L=\lfloor\frac {4M|J|}{d}\rfloor+1$. For each $t\in\{1,2,\ldots,L\}$ there is a $p_t\in\{0,1,\ldots,k\}$ with $|F^{(p_t)}(a_t)|>d$, thus either $\big|\,\text{Re}(\,F^{(p_t)}(a_t)\,)\big|>d/2$ or 
	$\big|\,\text{Im}(\,F^{(p_t)}(a_t)\,)\big|>d/2$. For the former case let $F_t: J\to\mathbb{R}$ be a $p_t$-times continuously differentiable function such that $\text{Re}(F)|_{J_t}\equiv F_t|_{J_t}$, $\parallel F_t\parallel_{J,\,p_t}=O_{M,\,|J|,\,p_t}(1)$ and $|F_t^{(p_t)}(x)|\ge d/4$ for each $x\in J$. Do the same thing for the other case. Thus we have the set $F_t^{-1}(-\rho,\rho)\subset J$ can be cover by $O_{k, M,|J|}(1/{d^{p_t}})$ intervals of length $\le2(\rho/d)^{1/{2^{p_t}}}$ each
	by \cite[Lemma 5.8]{hochman2014self}, which implies the corollary holds.
\end{proof}
The main result of this subsection is the following Theorem \ref{thmSeperation}.
\begin{theorem}\label{thmSeperation}
	There exists $\ell_0\in\mathbb{N}$ and $\varepsilon_0>0$ such that the following holds for  Lebesgue-a.e.$\,x\in [0,1]$. 
	
	For any integer $\ell\ge\ell_0$, there exists a set $Q_{x,\,\ell}\subset\mathbb{Z}_+$ such that 
	\begin{enumerate}
		\item [(i)] $\#Q_{x,\,\ell}=\infty $;
		\item [(ii)] for any $\textbf{w}\in\varLambda^{\ell}$, the sequence $(X_{n}^{\textbf{w},\,x})_{n\in\mathbb{Z}_+}$ is $(\varepsilon_0, Q_{x,\,\ell})$-exponential separation
	\end{enumerate}
	where $X_{n}^{\textbf{w},\,x}=\big\{S(x,\,\textbf{j} \textbf{w} )\,:\,\textbf{j}\in\varLambda^{n-\ell}\big\}$ for $n>\ell$ and $X_{n}^{\textbf{w},\,x}=\{0\}$ for $n\le\ell$.
\end{theorem}
\begin{proof}
The proof of \cite[Lemma 5.2]{ren2021dichotomy} works for all $\gamma\in\mathbb{C}$ such that $0<|\gamma|<1$. Combining this with Corollary \ref{lemHochmanSeperation2} we can prove  Theorem \ref{thmSeperation}  by the same method of \cite[Theorem 4.1]{Ren2022}, thus we omit the proof.
\end{proof}
\subsection{Transversality}\label{tra}
In this subsection we give some quantified estimates for transversality.
For any $\textbf{i},\,\textbf{j}\in\varLambda^{\#}\cup\Sigma$ and
integer $1\le k\le|\textbf{j}|$, recall $\textbf{j}_{k}=j_1j_2\ldots j_k$. 
Write $\textbf{i}<\textbf{j}$  if $\textbf{i}=\textbf{j}_{|\textbf{i}|}$ holds.
When $|\textbf{j}|<\infty$, let $I_{\textbf{j}}$ be an interval in $[0,1]$ such that
\begin{equation}\label{eq:interval}
I_{\textbf{j}}=\bigg[\frac{i_1+i_2b+\dots+i_nb^{n-1}}{b^n},\frac{1+i_1+i_2b+\dots+i_nb^{n-1}}{b^n}\bigg).
\end{equation}
The main result of this subsection is the following Theorem \ref{thmTransversality}.
\begin{theorem}\label{thmTransversality}
	For any $t_0>0$, there exists an integer $t>t_0$, real number $\Xi_1>0$ and  $\textbf{h},\,\textbf{h}',\,\textbf{a}\in\varLambda^t$ with  the following  property.
	For every $z\in I_{\textbf{a}}$ and $\textbf{i},\,\textbf{j}\in \varLambda^{\#}$, if $\textbf{h}<\textbf{i}$, 	$\textbf{h}'<\textbf{j}$, then
	\begin{enumerate}
		\item[(A.1)] $|S'(z,\,\textbf{i})|\,,\,|S'(z,\,\textbf{j})|>\Xi_1;$
		\item[(A.2)] $|S'(z,\,\textbf{i})-S'(z,\,\textbf{j})|>\Xi_1;$
		\item[(A.3)] $\frac{\parallel\phi'\parallel_{\infty}}{(1-\gamma)b^t}<\Xi_1/4.$
	\end{enumerate}
\end{theorem}
\begin{proof}
	Since \cite[Lemma 6.1, Lemma 6.2]{Ren2022} can be extended the case $\gamma\in\mathbb{C}$ such that $0<|\gamma|<1$, thus we could prove that (A.1) and (A.2) hold for some $t'>t_0$, $\Xi_1>0$ and  $\textbf{h},\,\textbf{h}',\,\textbf{a}\in\varLambda^{t'}$ by the same method of \cite[Theorem 6.1]{Ren2022}. Therefore it suffice to choose $t>t'$ large enough to make sure (A.3) also holds and replace $\textbf{h},\,\textbf{h}',\,\textbf{a}\in\varLambda^{t'}$ by $\textbf{h}\textbf{0}_{t-t'},\,\textbf{h}'\textbf{0}_{t-t'},\,\textbf{0}_{t-t'}\textbf{a}\in\varLambda^t$.
\end{proof}
\subsection{The partitions  of the space $\varLambda^{\#}$}\label{sec:partitionX}
This subsection is devoted to construct a  sequence of partitions $\mathcal{L}_n^{\varLambda^{\#}}$ of the space $\varLambda^{\#}$ by the same method of \cite[Section.7]{Ren2022}.
Combining Theorem \ref{TheoremB} with
Theorem \ref{thmSeperation} and Theorem \ref{thmTransversality},
there exists an integer $t>0$ such that (A.3) holds, some constants $\Xi_1,\, C>0$, a point $x_0\in [0,1)$, set $M\subset\mathbb{Z}_+$ and $\textbf{a},\,\textbf{h},\,\textbf{h}'\in\varLambda^t$  with the following properties.
\begin{enumerate}
	\item[(B.1)] For each  $\textbf{w}\in\varLambda^{t}$, the sequence $(X_n^{\textbf{w},\,x_0}\,)_{n\in\mathbb{Z}_+}$
	is $(|\gamma|^{C/2}, M)$-exponential separation
	where $X_n^{\textbf{w},\,x_0}$ is from Theorem \ref{thmSeperation};
	\item[(B.2)] $dim(m_{x_0})=\alpha$;
	\item[(B.3)]For every $z\in I_{\textbf{a}}$ and $\textbf{i},\,\textbf{j}\in \varLambda^{\#}$, if $\textbf{h}<\textbf{i}$, 	$\textbf{h}'<\textbf{j}$ then (A.1), (A.2) hold;
	\item[(B.4)] $\# M=\infty.$
\end{enumerate}
 We also  fix such elements $\{t,\,x_0,\,C,\, \Xi_1,\, M,\,\textbf{a},\,\textbf{h},\textbf{h}' \}$.
Let  $\overline{\pi}:\varLambda^{\#}\to\mathbb{N}\times\mathbb{C}^{3}$ be the map such that
$$\textbf{w}\mapsto\bigg(|\textbf{w}|,\,S(\textbf{w}(x_0),\textbf{h}),\,S(\textbf{w}(x_0),\textbf{h}'),\,S(x_0,\textbf{w})\bigg).$$
\begin{definition} For any integer $n\ge 1$, let 
	$\mathcal{L}_n^{\varLambda^{\#}}$ be the union of all the non-empty subsets of $\varLambda^{\#}$ of the following form
	$$\overline{\pi}^{\,-1}\left(\{m\}\times I_1 \times I_2  \times J\right),$$
	where $m\in \mathbb{N}, \,I_1,\,I_2\in \mathcal{L}_n,\, J\in \mathcal{L}^{\mathbb{C}}_{n+[m\log_b 1/|\gamma|]}.$
	The partition $\mathcal{L}_0^{\varLambda^{\#}}$ consists of non-empty subsets of $\varLambda^{\#}$ of the following form
	$$\overline{\pi}^{\,-1}\left(\{m\}\times\mathbb{C}\times\mathbb{C}\times J\right),$$
	where $m\in \mathbb{N}, J\in \mathcal{L}^{\mathbb{C}}_{[m\log_b 1/|\gamma|]}.$
\end{definition}
\subsection{The proof of Theorem A}
We shall recall the following notations from \cite{Ren2022}.  For any probability measure $\xi\in\mathscr{P}(\varLambda^{\#})$ and $\textbf{u}\in\varLambda^{\#}$, define the probability measure  $A_{\textbf{u}}(\xi)\in\mathscr{P}(\mathbb{C})$ be such that
\begin{equation}\label{measure}
A_{\textbf{u}}(\xi)=\sum_{\textbf{w}\in\text{supp}(\xi)}\xi(\{\textbf{w}\})\,\delta_{S(x_0,\textbf{w}\,\textbf{u})}.
\end{equation}
Let $B_{\textbf{q}}(\xi)\in\mathscr{P}(\mathbb{C})$ be the measure such that for any set $A\in\mathscr{B}(\mathbb{C})$, 
\begin{equation}\label{measure2}
B_{\textbf{q}}(\xi)(A)=\xi\times\nu^{\mathbb{Z}_+}\big(\,\big\{\,(\textbf{w},\, \textbf{j})\in\varLambda^{\#}\times \Sigma: S(x_0,\textbf{w}\,\textbf{q}\,\textbf{j})\in A\big\}\,\big).
\end{equation}
We also define the discrete measure on $\varLambda^{\#}$
\begin{equation}\label{measure3}
\theta^{\,\textbf{u}}_n:=\frac1{b^{\hat{n}-t}}\sum_{\textbf{w}\in\varLambda^{\hat{n}-t}}\delta_{\textbf{w}\textbf{u}}.
\end{equation}
\begin{proof}[The proof of Theorem A]
	\cite[Lemma 8.1, corollary 8.1 and Lemma 8.3]{Ren2022} can 
	be extended to the case $\gamma\in\mathbb{C}$ such that $0<|\gamma|<1$ by the same methods. Also by (B.1), (B.2) and (B.3), the proofs of  \cite[Lemma 7.1, Lemma 7.2, and Lemma 8.2, Lemma 8.4, Lemma 8.5 ]{Ren2022} also work for  the case $\gamma\in\mathbb{C}$ such that $0<|\gamma|<1$. Combining these with Theorem \ref{thm:entropygrowth}, we extend the conclusion of \cite[Lemma 8.6]{Ren2022} to the case $\gamma\in\mathbb{C}$  if $\alpha<\min\{2,\frac{\log b}{\log1/|\gamma|}\}$.
	Assume $\alpha<\min\{2,\frac{\log b}{\log1/|\gamma|}\}$,
	then there is a contradiction as we did in \cite{Ren2022}.
\end{proof}

\bibliographystyle{plain}             

\end{document}